\documentclass[11pt]{amsart}

\usepackage[T1]{fontenc}
\usepackage[utf8]{inputenc}
\usepackage{lmodern}
\usepackage{microtype}
\usepackage{amsmath,amssymb,amsthm,mathtools}
\usepackage{enumitem}
\usepackage[colorlinks=true,linkcolor=blue,citecolor=blue,urlcolor=blue]{hyperref}

\allowdisplaybreaks
\newcommand{\R}{\mathbb{R}}
\newcommand{\N}{\mathbb{N}}
\newcommand{\E}{\mathbb{E}}

\newcommand{\Vol}{\operatorname{Vol}}
\newcommand{\conv}{\operatorname{conv}}
\newcommand{\clconv}{\operatorname{cl\,conv}}

\newcommand{\Ord}{\operatorname{Ord}}
\newcommand{\1}{\mathbf{1}}
\newcommand{\dd}{\,\mathrm{d}}
\newcommand{\seg}[1]{[0,#1]}
\newcommand{\absdet}[1]{\left|\det\left[#1\right]\right|}
\newcommand{\ip}[2]{\langle #1,#2\rangle}
\newcommand{\IV}{\mathcal V}
\newcommand{\J}{\mathcal J}
\newcommand{\ff}[2]{(#1)_{\underline{#2}}}

\numberwithin{equation}{section}

\newtheorem{theorem}{Theorem}[section]
\newtheorem{proposition}[theorem]{Proposition}
\newtheorem{corollary}[theorem]{Corollary}
\newtheorem{lemma}[theorem]{Lemma}
\newtheorem{claim}[theorem]{Claim}
\theoremstyle{remark}
\newtheorem{remark}[theorem]{Remark}
\newtheorem{example}[theorem]{Example}

\title[Mixed volumes of random hulls]{Expected mixed volumes of convex hulls of random walks and L\'evy processes}

\author[A. S. Bolotin]{Artyom S. Bolotin}
\address{St. Petersburg Department of the Steklov Mathematical Institute of the Russian Academy of Sciences, 27 Fontanka, St. Petersburg 191023, Russia}
\email{bolotin2003@yandex.ru}

\author[D. N. Zaporozhets]{Dmitry N. Zaporozhets}
\address{St. Petersburg Department of the Steklov Mathematical Institute of the Russian Academy of Sciences, 27 Fontanka, St. Petersburg 191023, Russia}
\email{zap1979@gmail.com}

\date{September 2026}

\hypersetup{
  pdftitle={Expected mixed volumes of convex hulls of random walks and Levy processes},
  pdfauthor={Artem S. Bolotin and Dmitry N. Zaporozhets},
  pdfkeywords={random walk, Levy process, convex hull, mixed volume, valuation, orthoscheme}
}

\begin{document}

\begin{abstract}
Let $C_1,\ldots,C_k$ be the convex hulls of independent partial-sum processes in $\R^d$ whose increments are exchangeable within each process. We express the expected mixed volume $\E V_d(C_1[m_1],\ldots,C_k[m_k])$, $m_1+\cdots+m_k=d$, through mean absolute determinants of disjoint block sums of the increments; no general-position assumption is needed. For random walks with i.i.d.\ integrable increments the block sums are independent, and the formula extends the expected-volume formula of Barndorff-Nielsen and Baxter and of Vysotsky and Zaporozhets to mixed volumes. We then prove a continuous-time counterpart: for independent L\'evy processes with finite first moments, the expected mixed volume of the closed convex hulls of their paths is an explicit integral of mean absolute determinants over a product of simplices. For symmetric stable processes the integral is evaluated in terms of the associated zonoids. As a geometric application, we compute the mean mixed volume of random projections of mutually orthogonal canonical orthoschemes.
\end{abstract}

\maketitle

%======================================================================
\section{Introduction}\label{sec:intro}

The convex hull of the path of a random walk or of a stochastic process is a classical object at the interface of probability theory, stochastic geometry and convex geometry. In the plane, Spitzer and Widom \cite{SpitzerWidom} computed the expected perimeter of the convex hull of a random walk, and Baxter \cite{Baxter} gave a combinatorial proof of their formula; see \cite{MajumdarComtetRandonFurling} for a survey of planar results and their applications. Higher-dimensional versions were first considered by Barndorff-Nielsen and Baxter \cite{BarndorffBaxter}. Vysotsky and Zaporozhets \cite{VysotskyZaporozhets} expressed all expected intrinsic volumes of the convex hull of a random walk in $\R^d$ through Gram determinants of independent copies of the walk, and treated partial sums of exchangeable random vectors as well; see also \cite{KVZ-GAFA,KVZ-AdvMath} for absorption probabilities and the expected number of faces, and \cite{RandonFurlingZaporozhets} for convex hulls of several Gaussian random walks. For convex hulls of Brownian motions and L\'evy processes we refer to \cite{Eldan,KabluchkoZaporozhets,KampfLastMolchanov,MolchanovWespi,PitmanUribeBravo} and the references therein.

In particular, let $S_j=\xi_1+\cdots+\xi_j$ be a random walk in $\R^d$ with i.i.d.\ increments, and let
\[
 C_n=\conv(0,S_1,\ldots,S_n).
\]
Under a general-position assumption on the increments,
\begin{equation}\label{eq:known-volume}
 \E\Vol_d(C_n)
 =\frac1{d!}
 \sum_{\substack{j_1,\ldots,j_d\ge1\\j_1+\cdots+j_d\le n}}
 \frac{\E\absdet{S^{(1)}_{j_1},\ldots,S^{(d)}_{j_d}}}{j_1\cdots j_d},
\end{equation}
where $S^{(1)},\ldots,S^{(d)}$ are independent copies of the walk \cite[Corollary~2]{VysotskyZaporozhets}. As noted in \cite{VysotskyZaporozhets}, a version of this formula was first obtained by Barndorff-Nielsen and Baxter \cite{BarndorffBaxter}.

The purpose of the present paper is to develop a mixed-volume counterpart of \eqref{eq:known-volume}. Recall that the mixed volume $V_d(K_1,\ldots,K_d)$ of non-empty compact convex sets $K_1,\ldots,K_d\subset\R^d$ is the polarization of the volume with respect to Minkowski addition; see Section~\ref{sec:preliminaries}. Let $C_1,\ldots,C_k$ be the convex hulls of independent random walks in $\R^d$ with $n_1,\ldots,n_k$ steps and i.i.d.\ integrable increments, and let $m_1+\cdots+m_k=d$. Our first main result (Corollary~\ref{thm:discrete}) states that
\begin{equation}\label{eq:intro-discrete}
 \E V_d\bigl(C_1[m_1],\ldots,C_k[m_k]\bigr)
 =\frac1{d!}\sum
 \frac{\E\absdet{S^{(1,1)}_{j_{1,1}},\ldots,S^{(m_k,k)}_{j_{m_k,k}}}}
 {\prod_{i=1}^k\prod_{r=1}^{m_i}j_{r,i}},
\end{equation}
where $K[m]$ denotes $m$ copies of $K$, the sum runs over all positive integers $j_{r,i}$ such that $j_{1,i}+\cdots+j_{m_i,i}\le n_i$ for every $i$, and the columns of the determinant are the values at the times $j_{r,i}$ of mutually independent copies $S^{(r,i)}$ of the walks (the order of the columns is irrelevant). Formula \eqref{eq:intro-discrete} contains \eqref{eq:known-volume} as the case $k=1$, $m_1=d$, and it requires no general-position assumption. More generally, Theorem~\ref{thm:exchangeable} treats partial sums of exchangeable increments. In this case the independent copies in \eqref{eq:intro-discrete} must be replaced by the jointly distributed disjoint block sums of the increments, and Remark~\ref{rem:exchangeability} shows that this cannot be avoided.

The proof of the discrete results is combinatorial. Its geometric input is the classical staircase triangulation of a parallelotope into simplices indexed by permutations \cite{Freudenthal,Kuhn}. We determine the coefficients in the inclusion-exclusion formula over the chambers of this triangulation by counting intersections of chambers; this leads to recurrences and to identities for Stirling numbers (Lemma~\ref{lem:braid-valuation}). M\"obius inversion on the lattice of set partitions then yields a permutation-averaging identity valid for all valuations (Proposition~\ref{prop:permutation-valuation}), which is a valuation version of the classical combinatorial lemmas of Spitzer \cite{Spitzer}, Baxter \cite{Baxter} and Barndorff-Nielsen and Baxter \cite{BarndorffBaxter}. Applied to mixed-volume functionals and combined with exchangeability, it gives Theorem~\ref{thm:exchangeable} in general position, and a Gaussian lifting to a space of higher dimension removes this assumption. The discrete part of the paper is based on the first author's bachelor's thesis \cite{BolotinThesis}, written under the supervision of the second author.

Our second main result is a continuous-time version of \eqref{eq:intro-discrete}. Let $X_1,\ldots,X_k$ be independent L\'evy processes in $\R^d$ with finite first moments, and let
\[
 Z_i(T_i)=\clconv\{X_i(t):0\le t\le T_i\}.
\]
For an array $\mathbf t=(t_{r,i})$ let $\mathcal D_X(\mathbf t)$ be the absolute value of the determinant with columns $X_i^{(r)}(t_{r,i})$, $1\le r\le m_i$, $1\le i\le k$, where $X_i^{(r)}$ are mutually independent copies of $X_i$. Theorem~\ref{thm:continuous} states that
\begin{equation}\label{eq:intro-continuous}
 \E V_d\bigl(Z_1(T_1)[m_1],\ldots,Z_k(T_k)[m_k]\bigr)
 =\frac1{d!}
 \int_{\prod_{i=1}^k\Delta_{m_i}(T_i)}
 \frac{\E\mathcal D_X(\mathbf t)}
 {\prod_{i=1}^k\prod_{r=1}^{m_i}t_{r,i}}\,\dd\mathbf t,
\end{equation}
where
\[
 \Delta_m(T)=\{(t_1,\ldots,t_m)\in(0,\infty)^m:t_1+\cdots+t_m\le T\},
\]
and that the integral is finite. No symmetry, self-similarity or absolute continuity of the processes is required. The proof follows the approximation scheme of Molchanov and Wespi \cite{MolchanovWespi}, who computed the expected intrinsic volumes of the convex hull of a symmetric stable process by sampling the process at the times $i/n$ and applying the random-walk formulas of \cite{VysotskyZaporozhets}. For general L\'evy processes the resulting sums are not explicit, and we identify them as improper Riemann sums. The singularity at the origin is controlled by the elementary bound $\E\|X(t)\|\le c_T\sqrt t$, $0<t\le T$, which holds for every L\'evy process with finite first moment. For independent symmetric stable processes the integral in \eqref{eq:intro-continuous} is evaluated in terms of the associated zonoids (Corollary~\ref{cor:stable}); for a single process and $m_1=d$ this recovers the expected volume computed in \cite{MolchanovWespi}.

Finally, we apply the discrete theorem to canonical orthoschemes. Combining the Gaussian representation of intrinsic volumes due to Sudakov and Tsirelson \cite{Sudakov,Tsirelson} with Kubota's formula, we compute the mean mixed volume of the projections of mutually orthogonal canonical orthoschemes onto a random subspace (Theorem~\ref{thm:orthoscheme}). For a single orthoscheme this is equivalent to the formula of Gao and Vitale \cite{GaoVitale} for its intrinsic volumes; see also \cite{KabluchkoZaporozhets,VysotskyZaporozhets}.

The paper is organized as follows. Section~\ref{sec:preliminaries} collects the convex-geometric facts used throughout, and the main results are stated in Section~\ref{sec:main}. The permutation-averaging identity is proved in Section~\ref{sec:braid}, and the discrete theorems in Section~\ref{sec:discrete-proof}, where we also recover the known formulas for the expected intrinsic volumes of convex hulls of random walks (Corollary~\ref{cor:intrinsic}). Section~\ref{sec:continuous-proof} contains the proofs of the continuous-time results, and random projections of orthoschemes are treated in Section~\ref{sec:orthoschemes}.

%======================================================================
\section{Convex-geometric preliminaries}\label{sec:preliminaries}

\subsection{Mixed volumes and valuations}

Let $\mathcal K^d$ be the family of non-empty compact convex subsets of $\R^d$. We do not assume that these sets have interior points: mixed volumes and all their properties used below are defined on the whole family $\mathcal K^d$, whose elements are called convex bodies in \cite{Schneider}. This is essential for us, because many of the sets considered below, such as segments, parallelotopes spanned by fewer than $d$ vectors, and convex hulls of random walks with fewer than $d$ steps, have empty interior. By Minkowski's theorem \cite[Theorem~5.1.7]{Schneider}, for $K_1,\ldots,K_s\in\mathcal K^d$ and $\lambda_1,\ldots,\lambda_s\ge0$,
\begin{align}\label{eq:minkowski-polynomial}
 \Vol_d\left(\lambda_1K_1+\cdots+\lambda_sK_s\right)
 ={}&\sum_{\substack{r_1,\ldots,r_s\ge0\\r_1+\cdots+r_s=d}}
 \frac{d!}{r_1!\cdots r_s!}
 V_d\bigl(K_1[r_1],\ldots,K_s[r_s]\bigr)\\
 &\hspace{35mm}\times\lambda_1^{r_1}\cdots\lambda_s^{r_s}.\notag
\end{align}
Here $K[r]$ means that $K$ is repeated $r$ times, and arguments with multiplicity zero are omitted. The coefficient $V_d(K_1,\ldots,K_d)$ is the mixed volume, normalized so that $V_d(K[d])=\Vol_d(K)$. We shall use its standard properties \cite[Section~5.1]{Schneider}: symmetry, Minkowski multilinearity, translation invariance, monotonicity under inclusion in each argument, and continuity in the Hausdorff metric.

A real-valued function $\phi$ on the convex polytopes of a Euclidean space is called a valuation if
\begin{equation}\label{eq:valuation}
 \phi(A\cup B)+\phi(A\cap B)=\phi(A)+\phi(B)
\end{equation}
whenever $A$, $B$ and $A\cup B$ are convex polytopes; note that then $A\cap B\neq\varnothing$. If $L_1,\ldots,L_{d-m}\in\mathcal K^d$ are fixed, then
\begin{equation}\label{eq:diagonal-valuation}
 K\longmapsto V_d\bigl(L_1,\ldots,L_{d-m},K[m]\bigr)
\end{equation}
is a continuous translation-invariant valuation on $\mathcal K^d$; see, for example, \cite[Chapters~5 and~6]{Schneider}. By Volland's extension theorem \cite{Volland} (see also \cite{Groemer}, \cite[Chapter~6]{Schneider} and \cite{KlainRota}), every valuation on convex polytopes admits an additive extension to finite unions of convex polytopes. In particular, the inclusion-exclusion principle
\begin{equation}\label{eq:inclusion-exclusion}
 \phi\left(P_1\cup\cdots\cup P_N\right)
 =\sum_{\varnothing\neq I\subset\{1,\ldots,N\}}(-1)^{|I|+1}
 \phi\Bigl(\bigcap_{i\in I}P_i\Bigr)
\end{equation}
holds whenever $P_1,\ldots,P_N$ and their union are convex polytopes and all intersections on the right-hand side are non-empty. It will be applied to the chambers of a triangulation in Section~\ref{sec:braid}.

\subsection{Segments and projections}

For vectors $x_1,\ldots,x_d\in\R^d$,
\begin{equation}\label{eq:segments-determinant}
 V_d\bigl(\seg{x_1},\ldots,\seg{x_d}\bigr)
 =\frac1{d!}\absdet{x_1,\ldots,x_d}.
\end{equation}
For a linear subspace $E$ we denote by $K|E$ the orthogonal projection of $K$ onto $E$. If $e$ is a unit vector in $\R^r$ and $K_1,\ldots,K_{r-1}\in\mathcal K^r$, then
\begin{equation}\label{eq:segment-projection}
 r\,V_r\bigl(K_1,\ldots,K_{r-1},\seg e\bigr)
 =V_{r-1}\bigl(K_1|e^\perp,\ldots,K_{r-1}|e^\perp\bigr),
\end{equation}
where the right-hand side is the mixed volume in the $(r-1)$-dimensional space $e^\perp$; see, for example, \cite[Chapter~5]{Schneider}. Iterating \eqref{eq:segment-projection}, we obtain the following. If $E$ is a $d$-dimensional linear subspace of $\R^q$, $f_1,\ldots,f_{q-d}$ is an orthonormal basis of $E^\perp$ and $K_1,\ldots,K_d\in\mathcal K^q$, then
\begin{equation}\label{eq:iterated-projection}
 V_q\bigl(K_1,\ldots,K_d,
 \seg{f_1},\ldots,\seg{f_{q-d}}\bigr)
 =\frac{d!}{q!}V_d\bigl(K_1|E,\ldots,K_d|E\bigr).
\end{equation}

\subsection{Random segments and selection expectations}

A random compact convex set $K$ in $\R^d$ is called integrably bounded if $\E\max_{x\in K}\|x\|<\infty$. Its selection expectation $\E K$ is the set in $\mathcal K^d$ whose support function is $h(\E K,u)=\E h(K,u)$, $u\in\R^d$; see \cite{MolchanovRandomSets}. In particular, for an integrable random vector $\xi$,
\[
 h(\E\seg\xi,u)=\E\max\{0,\ip u\xi\},\qquad u\in\R^d.
\]
If $\xi_1,\ldots,\xi_d$ are independent integrable random vectors in $\R^d$, then
\begin{equation}\label{eq:random-segments}
 V_d\bigl(\E\seg{\xi_1},\ldots,\E\seg{\xi_d}\bigr)
 =\E V_d\bigl(\seg{\xi_1},\ldots,\seg{\xi_d}\bigr)
 =\frac1{d!}\E\absdet{\xi_1,\ldots,\xi_d}.
\end{equation}
This is a result of Vitale \cite{Vitale}; see also \cite[Theorem~2.1]{MolchanovWespi}. For the reader's convenience we recall the argument. For fixed $K_2,\ldots,K_d\in\mathcal K^d$,
\[
 V_d(K_1,K_2,\ldots,K_d)=\frac1d\int_{S^{d-1}}h(K_1,u)\,S(K_2,\ldots,K_d;\dd u),
\]
where $S(K_2,\ldots,K_d;\cdot)$ is the mixed area measure \cite[Section~5.1]{Schneider}. Since $0\le h(\seg{\xi_1},u)\le\|\xi_1\|$ for unit vectors $u$, Fubini's theorem gives $\E V_d(\seg{\xi_1},K_2,\ldots,K_d)=V_d(\E\seg{\xi_1},K_2,\ldots,K_d)$. Applying this to each argument in turn, conditionally on the remaining arguments and using independence, we obtain the first equality in \eqref{eq:random-segments}. The second equality follows from \eqref{eq:segments-determinant}.

\subsection{Stable processes and associated zonoids}

Let $X$ be a symmetric $\alpha$-stable L\'evy process in $\R^d$, $1<\alpha\le2$. There is a set $K\in\mathcal K^d$, called the associated zonoid of $X$, such that
\begin{equation}\label{eq:stable-zonoid}
 \E e^{i\ip u{X(t)}}=\exp\{-t\,h(K,u)^\alpha\},\qquad u\in\R^d,\ t\ge0;
\end{equation}
see \cite{KampfLastMolchanov,MolchanovStable}. Indeed, for every $u$ the random variable $\ip u{X(1)}$ is symmetric $\alpha$-stable, and a symmetric $\alpha$-stable random variable $Y$ with $\E e^{isY}=e^{-|s|^\alpha}$ satisfies $\E|Y|=2\Gamma(1-1/\alpha)/\pi$. Since $h(\E[-X(1),X(1)],u)=\E|\ip u{X(1)}|$, it follows that \eqref{eq:stable-zonoid} holds with the zonoid
\[
 K=\frac{\pi}{2\Gamma(1-1/\alpha)}\,\E[-X(1),X(1)].
\]
Moreover, $h(\E\seg{X(1)},u)=\E\max\{0,\ip u{X(1)}\}=\tfrac12\E|\ip u{X(1)}|$ by symmetry, and therefore
\begin{equation}\label{eq:stable-selection}
 \E\seg{X(1)}=\frac{\Gamma(1-1/\alpha)}{\pi}\,K.
\end{equation}

\subsection{Index sets}

Let $\N=\{1,2,\ldots\}$. For $n,m\in\N$ define
\begin{equation}\label{eq:Jnm}
 \J_m(n)=\left\{(j_1,\ldots,j_m)\in\N^m:j_1+\cdots+j_m\le n\right\},
\end{equation}
and put $\J_0(n)=\{\varnothing\}$. For $T>0$ define
\begin{equation}\label{eq:simplex}
 \Delta_m(T)=\left\{(t_1,\ldots,t_m)\in(0,\infty)^m:t_1+\cdots+t_m\le T\right\},
\end{equation}
and regard $\Delta_0(T)$ as a one-point space with integral equal to one.

%======================================================================
\section{Main results}\label{sec:main}

\subsection{Exchangeable increments and random walks}\label{ssec:discrete-results}

For $i=1,\ldots,k$, let
\[
 \boldsymbol\xi_i=(\xi_{1,i},\ldots,\xi_{n_i,i})
\]
be a finite exchangeable sequence of random vectors in $\R^d$; that is, $(\xi_{\sigma(1),i},\ldots,\xi_{\sigma(n_i),i})$ has the same distribution as $\boldsymbol\xi_i$ for every permutation $\sigma$ of $\{1,\ldots,n_i\}$. The $k$ sequences are assumed to be independent. Put
\[
 S_{j,i}=\xi_{1,i}+\cdots+\xi_{j,i},\qquad j=0,1,\ldots,n_i,
 \qquad
 C_i=\conv(0,S_{1,i},\ldots,S_{n_i,i}).
\]
For $\mathbf j_i=(j_{1,i},\ldots,j_{m_i,i})\in\J_{m_i}(n_i)$, write
\[
 a_{0,i}=0,\qquad a_{r,i}=j_{1,i}+\cdots+j_{r,i},
\]
and define the consecutive disjoint block sums
\begin{equation}\label{eq:block-sums}
 Y_{\mathbf j_i}^{(r,i)}
 =S_{a_{r,i},i}-S_{a_{r-1,i},i}
 =\sum_{\ell=a_{r-1,i}+1}^{a_{r,i}}\xi_{\ell,i},
 \qquad r=1,\ldots,m_i.
\end{equation}

\begin{theorem}[Exchangeable increments]\label{thm:exchangeable}
Let $m_1,\ldots,m_k$ be non-negative integers with $m_1+\cdots+m_k=d$. Then, with expectations understood in $[0,\infty]$,
\begin{align}\label{eq:exchangeable-main}
 &\E V_d\bigl(C_1[m_1],\ldots,C_k[m_k]\bigr)\\
 &\quad=\frac1{d!}
 \sum_{\substack{\mathbf j_i\in\J_{m_i}(n_i)\\i=1,\ldots,k}}
 \frac{
 \E\absdet{
 Y_{\mathbf j_1}^{(1,1)},\ldots,Y_{\mathbf j_1}^{(m_1,1)},
 \ldots,
 Y_{\mathbf j_k}^{(1,k)},\ldots,Y_{\mathbf j_k}^{(m_k,k)}
 }}
 {\displaystyle\prod_{i=1}^k\prod_{r=1}^{m_i}j_{r,i}}.\notag
\end{align}
No general-position assumption is required.
\end{theorem}

If $m_i>n_i$ for some $i$, then both sides of \eqref{eq:exchangeable-main} vanish: the sum is empty, and $C_i$, being the convex hull of $n_i+1$ points, has dimension less than $m_i$. Both sides of \eqref{eq:exchangeable-main} may be infinite; a simple sufficient condition for their finiteness is given in \eqref{eq:finiteness} below.

For $k=1$ and $m_1=d$, Theorem~\ref{thm:exchangeable} is an expected-volume formula for partial sums of exchangeable random vectors. Under a general-position assumption, this case can also be derived from \cite[Theorem~4]{VysotskyZaporozhets}, which is formulated for exchangeable increments, by means of the lifting argument used in the proof of \cite[Corollary~2]{VysotskyZaporozhets}.

For ordinary random walks the block sums in \eqref{eq:block-sums} are independent.

\begin{corollary}[Independent-copy formula]\label{thm:discrete}
Suppose, in addition, that for each $i$ the increments $\xi_{1,i},\ldots,\xi_{n_i,i}$ are i.i.d.\ and integrable. For every $i$ and $r\ge1$, let $S^{(r,i)}$ be a copy of the $i$-th walk, all copies being mutually independent. Then all quantities below are finite and
\begin{align}\label{eq:discrete-main}
 &\E V_d\bigl(C_1[m_1],\ldots,C_k[m_k]\bigr)\\
 &\quad=\frac1{d!}
 \sum_{\substack{\mathbf j_i\in\J_{m_i}(n_i)\\i=1,\ldots,k}}
 \frac{
 \E\absdet{
 S^{(1,1)}_{j_{1,1}},\ldots,S^{(m_1,1)}_{j_{m_1,1}},
 \ldots,
 S^{(1,k)}_{j_{1,k}},\ldots,S^{(m_k,k)}_{j_{m_k,k}}
 }}
 {\displaystyle\prod_{i=1}^k\prod_{r=1}^{m_i}j_{r,i}}.\notag
\end{align}
\end{corollary}

For $k=1$ and $m_1=d$, Corollary~\ref{thm:discrete} is the expected-volume formula \eqref{eq:known-volume}, now without any general-position assumption.

\begin{remark}[Exchangeability versus independent copies]\label{rem:exchangeability}
The distinction between Theorem~\ref{thm:exchangeable} and Corollary~\ref{thm:discrete} is essential. Let $n=d=2$ and let $(\xi_1,\xi_2)$ be equally likely to be $(e_1,e_2)$ or $(e_2,e_1)$. The increments are exchangeable and linearly independent almost surely, while
\[
 \E\Vol_2\bigl(\conv(0,\xi_1,\xi_1+\xi_2)\bigr)=\frac12.
\]
If $\xi_1^{(1)},\xi_1^{(2)}$ are independent copies of the common marginal, then
\[
 \frac12\E\absdet{\xi_1^{(1)},\xi_1^{(2)}}=\frac14.
\]
Thus, under mere exchangeability, the jointly distributed block sums in \eqref{eq:exchangeable-main} cannot in general be replaced by independent copies.
\end{remark}

Corollary~\ref{thm:discrete} and \eqref{eq:minkowski-polynomial} give the expected volume of a Minkowski sum of random-walk hulls.

\begin{corollary}\label{cor:discrete-minkowski}
Under the assumptions of Corollary~\ref{thm:discrete}, for $\lambda_1,\ldots,\lambda_k\ge0$,
\begin{align}\label{eq:discrete-minkowski}
 \E\Vol_d\left(\lambda_1C_1+\cdots+\lambda_kC_k\right)
 ={}&\sum_{\substack{m_1,\ldots,m_k\ge0\\m_1+\cdots+m_k=d}}
 \frac{\lambda_1^{m_1}\cdots\lambda_k^{m_k}}{m_1!\cdots m_k!}\\
 &\times
 \sum_{\substack{\mathbf j_i\in\J_{m_i}(n_i)\\i=1,\ldots,k}}
 \frac{\E\absdet{S^{(1,1)}_{j_{1,1}},\ldots,S^{(m_k,k)}_{j_{m_k,k}}}}
 {\displaystyle\prod_{i=1}^k\prod_{r=1}^{m_i}j_{r,i}}.\notag
\end{align}
\end{corollary}

\subsection{Independent L\'evy processes}

Let $X_i=(X_i(t))_{t\ge0}$, $i=1,\ldots,k$, be independent L\'evy processes in $\R^d$, always taken in their c\`adl\`ag versions. Assume
\begin{equation}\label{eq:levy-first-moment}
 \E\|X_i(1)\|<\infty,\qquad i=1,\ldots,k.
\end{equation}
For $T_i>0$ put
\begin{equation}\label{eq:levy-hulls}
 Z_i(T_i)=\clconv\{X_i(t):0\le t\le T_i\}.
\end{equation}
For every $i$ and $r\ge1$, let $X_i^{(r)}$ be a copy of $X_i$, all copies being mutually independent. For an array of times $\mathbf t=(t_{r,i})$, put
\begin{equation}\label{eq:det-process-array}
 \mathcal D_X(\mathbf t)
 =\left|\det\left[ X_i^{(r)}(t_{r,i})\right]_{
 \substack{1\le i\le k\\1\le r\le m_i}}\right|,
\end{equation}
where the columns are ordered lexicographically in $(i,r)$.

\begin{theorem}[Continuous-time mixed-volume formula]\label{thm:continuous}
Let $m_1,\ldots,m_k$ be non-negative integers with $m_1+\cdots+m_k=d$. Under \eqref{eq:levy-first-moment},
\begin{align}\label{eq:continuous-main}
 &\E V_d\bigl(Z_1(T_1)[m_1],\ldots,Z_k(T_k)[m_k]\bigr)\\
 &\quad=\frac1{d!}
 \int_{\prod_{i=1}^k\Delta_{m_i}(T_i)}
 \frac{\E\mathcal D_X(\mathbf t)}
 {\displaystyle\prod_{i=1}^k\prod_{r=1}^{m_i}t_{r,i}}
 \,\dd\mathbf t.\notag
\end{align}
The integral is finite.
\end{theorem}

For $k=1$ and $m_1=d$, Theorem~\ref{thm:continuous} expresses the expected volume of the convex hull of a single L\'evy process,
\[
 \E\Vol_d\bigl(Z(T)\bigr)=\frac1{d!}\int_{\Delta_d(T)}
 \frac{\E\absdet{X^{(1)}(t_1),\ldots,X^{(d)}(t_d)}}{t_1\cdots t_d}\,\dd t_1\cdots\dd t_d,
\]
which is the continuous-time analogue of \eqref{eq:known-volume}. Note that the finiteness assertion of Theorem~\ref{thm:continuous} requires only the first-moment assumption \eqref{eq:levy-first-moment}; compare \cite[Theorem~1.1]{MolchanovWespi}.

There is a useful version with independent exponential horizons. Let $\tau_i$ be exponentially distributed with parameter $q_i>0$, independently of all processes and of one another.

\begin{corollary}[Independent exponential horizons]\label{cor:exponential}
Under the assumptions of Theorem~\ref{thm:continuous},
\begin{align}\label{eq:exponential-main}
 &\E V_d\bigl(Z_1(\tau_1)[m_1],\ldots,Z_k(\tau_k)[m_k]\bigr)\\
 &\quad=\frac1{d!}\int_{(0,\infty)^d}
 \exp\left\{-\sum_{i=1}^k q_i\sum_{r=1}^{m_i}t_{r,i}\right\}
 \frac{\E\mathcal D_X(\mathbf t)}
 {\displaystyle\prod_{i=1}^k\prod_{r=1}^{m_i}t_{r,i}}
 \,\dd\mathbf t,\notag
\end{align}
and the right-hand side is finite.
\end{corollary}

The stable case can be evaluated explicitly.

\begin{corollary}[Independent symmetric stable processes]\label{cor:stable}
Suppose that $X_i$ is symmetric $\alpha_i$-stable with $1<\alpha_i\le2$, and let $K_i$ be its associated zonoid, so that
\[
 \E e^{i\ip u{X_i(t)}}=\exp\{-t\,h(K_i,u)^{\alpha_i}\}.
\]
Then
\begin{align}\label{eq:stable-main}
 &\E V_d\bigl(Z_1(T_1)[m_1],\ldots,Z_k(T_k)[m_k]\bigr)\\
 &\quad=\frac{V_d\bigl(K_1[m_1],\ldots,K_k[m_k]\bigr)}{\pi^d}
 \prod_{i=1}^k
 \frac{\bigl[\Gamma(1/\alpha_i)\Gamma(1-1/\alpha_i)\bigr]^{m_i}}
 {\Gamma(1+m_i/\alpha_i)}
 T_i^{m_i/\alpha_i}.\notag
\end{align}
\end{corollary}

For $k=1$, $m_1=d$ and $T_1=1$, formula \eqref{eq:stable-main} is the case $j=d$ of \cite[Theorem~2.3]{MolchanovWespi}, where all expected intrinsic volumes of $Z_1(1)$ were computed.

\begin{example}[Brownian motions]\label{ex:brownian}
Let $X_1,\ldots,X_k$ be independent standard Brownian motions in $\R^d$. Then $\alpha_i=2$ and $K_i=2^{-1/2}B^d$, where $B^d$ is the Euclidean unit ball. Hence
\begin{equation}\label{eq:brownian-main}
 \E V_d\bigl(Z_1(T_1)[m_1],\ldots,Z_k(T_k)[m_k]\bigr)
 =\frac{\kappa_d}{2^{d/2}}
 \prod_{i=1}^k\frac{T_i^{m_i/2}}{\Gamma(1+m_i/2)},
\end{equation}
where $\kappa_d=\Vol_d(B^d)=\pi^{d/2}/\Gamma(1+d/2)$. For $k=1$ this gives $\E\Vol_d(Z_1(T))=(\pi/2)^{d/2}\,T^{d/2}/\Gamma(1+d/2)^2$, which is Eldan's formula \cite{Eldan}. For two planar Brownian motions at time one,
\[
 \E\Vol_2(Z_1(1))=\frac\pi2,
 \qquad
 \E V_2(Z_1(1),Z_2(1))=2.
\]
The first value is due to El Bachir \cite{ElBachir}. Both values agree with \cite{MajumdarComtetRandonFurling,Dospolova}, where they were obtained by methods specific to the plane.
\end{example}

\begin{example}[Isotropic stable processes]\label{ex:isotropic}
If the distribution of $X_i(1)$ is rotation invariant, then $\E e^{i\ip u{X_i(1)}}=\exp\{-c_i\|u\|^{\alpha_i}\}$ for some $c_i>0$ \cite[Theorem~14.14]{Sato}, so that $K_i=c_i^{1/\alpha_i}B^d$. In this case \eqref{eq:stable-main} becomes
\begin{align*}
 &\E V_d\bigl(Z_1(T_1)[m_1],\ldots,Z_k(T_k)[m_k]\bigr)\\
 &\quad=\frac{1}{\pi^{d/2}\,\Gamma(1+d/2)}
 \prod_{i=1}^k
 \frac{\bigl[\Gamma(1/\alpha_i)\Gamma(1-1/\alpha_i)\,c_i^{1/\alpha_i}\bigr]^{m_i}}
 {\Gamma(1+m_i/\alpha_i)}\,T_i^{m_i/\alpha_i}.
\end{align*}
\end{example}

\subsection{Random projections of orthoschemes}\label{ssec:orthoscheme-results}

Let $n_1+\cdots+n_k=n$, and write the standard orthonormal basis of $\R^n$ in mutually orthogonal blocks
\[
 e_{1,1},\ldots,e_{n_1,1};\quad\ldots;\quad e_{1,k},\ldots,e_{n_k,k}.
\]
Define the canonical orthoschemes
\begin{equation}\label{eq:orthoschemes}
 \Omega_i=\conv\bigl(0,e_{1,i},e_{1,i}+e_{2,i},\ldots,e_{1,i}+\cdots+e_{n_i,i}\bigr),
 \qquad i=1,\ldots,k.
\end{equation}
Let $1\le d\le n$, and let $L$ be a random $d$-dimensional linear subspace of $\R^n$ with the Haar (rotation-invariant) distribution.

\begin{theorem}[Mixed volumes of random orthoscheme projections]\label{thm:orthoscheme}
If $m_1+\cdots+m_k=d$, then
\begin{align}\label{eq:orthoscheme-main}
 &\E_L V_d\bigl((\Omega_1|L)[m_1],\ldots,(\Omega_k|L)[m_k]\bigr)\\
 &\quad=\frac{(n-d)!\,\kappa_{n-d}\kappa_d}{n!\,\kappa_n}
 \sum_{\substack{\mathbf j_i\in\J_{m_i}(n_i)\\i=1,\ldots,k}}
 \frac1{\sqrt{\displaystyle\prod_{i=1}^k\prod_{r=1}^{m_i}j_{r,i}}}.\notag
\end{align}
\end{theorem}

For $k=1$ the left-hand side of \eqref{eq:orthoscheme-main} is $\E_L\Vol_d(\Omega_1|L)$, and by Kubota's formula \eqref{eq:kubota} below, \eqref{eq:orthoscheme-main} is equivalent to the formula
\[
 \IV_d(\Omega)=\frac1{d!}\sum_{(j_1,\ldots,j_d)\in\J_d(n)}\frac1{\sqrt{j_1\cdots j_d}}
\]
of Gao and Vitale \cite{GaoVitale} for the intrinsic volumes of the canonical orthoscheme $\Omega\subset\R^n$; see also \cite[Section~4.1]{VysotskyZaporozhets}.

%======================================================================
\section{Staircase triangulations and partition inversion}\label{sec:braid}

This section contains the deterministic valuation identity behind the discrete results. The argument follows \cite[Sections~6 and~7]{BolotinThesis}.

Let $[n]=\{1,\ldots,n\}$, and let $\Pi_n$ be the lattice of set partitions of $[n]$, ordered by refinement: $\pi\le\rho$ if every block of $\pi$ is contained in a block of $\rho$. Write $|\pi|$ for the number of blocks of $\pi$, and let $\pi_0=\{\{1\},\ldots,\{n\}\}$ be the finest partition.

Let $x_1,\ldots,x_n$ be vectors in a Euclidean space. For $B\subset[n]$ put $x_B=\sum_{j\in B}x_j$, and for $\pi\in\Pi_n$ define the parallelotope
\begin{equation}\label{eq:Ppi}
 P_\pi=\sum_{B\in\pi}\seg{x_B}.
\end{equation}
If $\rho\ge\pi$ and $\tau=(B_1,\ldots,B_\ell)$ is an ordering of the blocks of $\rho$, put
\begin{equation}\label{eq:Ktau}
 K_\tau=\conv\bigl(0,x_{B_1},x_{B_1}+x_{B_2},\ldots,x_{B_1}+\cdots+x_{B_\ell}\bigr).
\end{equation}
We write $\Ord(\rho)$ for the set of all $|\rho|!$ orderings of the blocks of $\rho$.

\begin{lemma}[Interior-face formula]\label{lem:braid-valuation}
Let $x_1,\ldots,x_n$ be linearly independent, and let $\phi$ be a valuation on convex polytopes. Then, for every $\pi\in\Pi_n$,
\begin{equation}\label{eq:braid-valuation}
 \phi(P_\pi)
 =\sum_{\rho\ge\pi}(-1)^{|\pi|-|\rho|}
 \sum_{\tau\in\Ord(\rho)}\phi(K_\tau).
\end{equation}
\end{lemma}

The simplices $K_\tau$ in \eqref{eq:braid-valuation} are the faces of the staircase triangulation of $P_\pi$ described below that contain both $0$ and $x_{[n]}$. The proof of Lemma~\ref{lem:braid-valuation} occupies Sections~\ref{ssec:chambers} and~\ref{ssec:proof-lemma}.

\subsection{Chambers and their intersections}\label{ssec:chambers}

Fix $\pi\in\Pi_n$, put $r=|\pi|$, and enumerate the vectors $x_B$, $B\in\pi$, as $y_1,\ldots,y_r$. They are linearly independent, and every point of $P=P_\pi$ can be written uniquely as $u_1y_1+\cdots+u_ry_r$ with $u=(u_1,\ldots,u_r)\in[0,1]^r$. We identify $P$ with the cube $[0,1]^r$ via these coordinates. For a permutation $\sigma$ of $[r]$ let
\begin{align*}
 K_\sigma&=\conv\bigl(0,y_{\sigma(1)},y_{\sigma(1)}+y_{\sigma(2)},\ldots,y_{\sigma(1)}+\cdots+y_{\sigma(r)}\bigr)\\
 &=\{u:0\le u_{\sigma(r)}\le\cdots\le u_{\sigma(1)}\le1\}.
\end{align*}
These $r!$ simplices, which we call chambers, form the classical staircase triangulation of the cube \cite{Freudenthal,Kuhn}; in particular, $P=\bigcup_{\sigma}K_\sigma$. We write $\mathfrak S_r$ for the set of permutations of $[r]$ and $\mathrm{id}$ for the identity permutation.

Every intersection of chambers contains $0$ and $y_1+\cdots+y_r$ and is a common face of the chambers involved. It is obtained from any chamber $K_\sigma$ containing it by imposing equalities within consecutive groups of coordinates in the order of $\sigma$. If the groups have sizes $p_1,\ldots,p_s$, where $p_1+\cdots+p_s=r$, then the intersection equals
\[
 \conv\Bigl(0,\ \sum_{a=1}^{p_1}y_{\sigma(a)},\ \sum_{a=1}^{p_1+p_2}y_{\sigma(a)},\ \ldots,\ \sum_{a=1}^{r}y_{\sigma(a)}\Bigr),
\]
and we say that it has type $[p_1,\ldots,p_s]$. The type is determined by the intersection itself, namely by the chain of sets $\{\sigma(1),\ldots,\sigma(p_1)\}\subset\{\sigma(1),\ldots,\sigma(p_1+p_2)\}\subset\cdots$. There are $r!/(p_1!\cdots p_s!)$ distinct intersections of type $[p_1,\ldots,p_s]$, and the chambers themselves are the intersections of type $[1,\ldots,1]$.

\begin{lemma}\label{lem:counting}
For $m\ge1$ and a type $[p_1,\ldots,p_s]$, let $c_m(p_1,\ldots,p_s)$ be the number of ordered $m$-tuples $(\sigma_1,\ldots,\sigma_m)$ of permutations of $[r]$, repetitions and the identity being allowed, such that $K_{\sigma_1}\cap\cdots\cap K_{\sigma_m}\cap K_{\mathrm{id}}$ has type $[p_1,\ldots,p_s]$. Then
\begin{equation}\label{eq:c-multiplicative}
 c_m(p_1,\ldots,p_s)=c_m(p_1)\cdots c_m(p_s),
\end{equation}
where $c_m(p)$ denotes the corresponding number for the type $[p]$ in the cube $[0,1]^p$. Moreover, $c_m(1)=1$ and
\begin{equation}\label{eq:c-recurrence}
 c_m(p)=(p!)^m-\sum_{j=1}^{p-1}c_m(j)\bigl((p-j)!\bigr)^m,\qquad p\ge2.
\end{equation}
\end{lemma}

\begin{proof}
Let $G_1=\{1,\ldots,p_1\}$, $G_2=\{p_1+1,\ldots,p_1+p_2\}$, and so on. The intersection $K_{\sigma_1}\cap\cdots\cap K_{\sigma_m}\cap K_{\mathrm{id}}$ has type $[p_1,\ldots,p_s]$ if and only if two conditions hold. First, every $\sigma_\ell$ maps each of the sets $G_1\cup\cdots\cup G_a$, $a=1,\ldots,s$, onto itself. Second, for every $a$ the restrictions of $\sigma_1,\ldots,\sigma_m$ to $G_a$, together with the identity, produce a single group of size $p_a$. The restrictions to different groups can be chosen independently, which proves \eqref{eq:c-multiplicative}.

The only permutation of $\{1\}$ is the identity, so $c_m(1)=1$. Let $p\ge2$. There are $(p!)^m$ tuples of permutations of $[p]$. If the intersection of a tuple with $K_{\mathrm{id}}$ has more than one group, then its first group is $\{1,\ldots,j\}$ for some $j<p$. Given $j$, the restrictions of the permutations to $\{1,\ldots,j\}$ can be chosen in $c_m(j)$ ways, and their restrictions to the remaining $p-j$ elements in $((p-j)!)^m$ ways. This proves \eqref{eq:c-recurrence}.
\end{proof}

\begin{remark}\label{rem:generating}
In terms of the formal power series
\[
 F_m(w)=\sum_{j=0}^{\infty}(j!)^m w^j,
 \qquad \mathcal P_m(w)=\sum_{p=1}^{\infty}c_m(p)\,w^p,
\]
the recurrence \eqref{eq:c-recurrence} reads $\mathcal P_m(w)F_m(w)=F_m(w)-1$, that is,
\[
 \mathcal P_m(w)=1-\frac1{F_m(w)}.
\]
For $m=1$, the permutations counted by $c_1(p)$ are exactly those that do not map any proper initial segment $\{1,\ldots,j\}$, $1\le j<p$, onto itself. Thus $(c_1(p))_{p\ge1}=(1,1,3,13,71,461,\ldots)$ is the sequence of the numbers of indecomposable permutations (sequence A003319 in the OEIS).
\end{remark}

Fix now a type $[p_1,\ldots,p_s]$ in which at least one $p_a$ exceeds one, and write $c_m=c_m(p_1,\ldots,p_s)$. Let $\widetilde c_m=\widetilde c_m(p_1,\ldots,p_s)$ be the number of $m$-element sets $\{\sigma_1,\ldots,\sigma_m\}$ of distinct non-identity permutations such that $K_{\sigma_1}\cap\cdots\cap K_{\sigma_m}\cap K_{\mathrm{id}}$ has type $[p_1,\ldots,p_s]$. Let $S(a,b)$ denote the Stirling numbers of the second kind and $s(a,b)$ the signed Stirling numbers of the first kind.

\begin{lemma}\label{lem:stirling}
For every $m\ge1$,
\begin{equation}\label{eq:c-ctilde}
 c_m=\sum_{j=1}^{m}j!\,S(m+1,j+1)\,\widetilde c_j
 \qquad\text{and}\qquad
 m!\,\widetilde c_m=\sum_{q=1}^{m}s(m+1,q+1)\,c_q .
\end{equation}
Consequently, for every $M\ge2$,
\begin{equation}\label{eq:alternating-stirling}
 \sum_{j=2}^{M}\frac{(-1)^{j+1}}{j}\,\widetilde c_{j-1}
 =\frac{(-1)^{M-1}}{M!}\sum_{j=3}^{M+1}s(M+1,j)\,c_{j-2}.
\end{equation}
\end{lemma}

\begin{proof}
Since the type has a part exceeding one, every tuple counted by $c_m$ contains a non-identity permutation. Suppose that such a tuple contains exactly $j$ distinct non-identity permutations; its intersection with $K_{\mathrm{id}}$ depends only on the set of these permutations. Add to the $m$ positions of the tuple one more position, reserved for the identity. The $m+1$ positions then split into $j+1$ non-empty classes. The class of the additional position consists of it and of the positions occupied by the identity, and each of the other $j$ classes consists of the positions occupied by one of the non-identity permutations. There are $S(m+1,j+1)$ such splittings, $j!$ ways to assign the $j$ permutations to the classes that do not contain the additional position, and $\widetilde c_j$ ways to choose the set of permutations. This proves the first identity in \eqref{eq:c-ctilde}.

The lower-triangular matrices $\bigl(S(a+1,b+1)\bigr)_{a,b=1}^{m}$ and $\bigl(s(a+1,b+1)\bigr)_{a,b=1}^{m}$ are mutually inverse, because $\sum_{b}S(a,b)\,s(b,c)=\delta_{ac}$ and $s(1,c)=0$ for $c\ge2$; see \cite[Chapter~6]{GrahamKnuthPatashnik}. Hence the first identity in \eqref{eq:c-ctilde} implies the second one.

By the second identity in \eqref{eq:c-ctilde}, $\widetilde c_{j-1}/j=(j-1)!\,\widetilde c_{j-1}/j!=\sum_{q=2}^{j}s(j,q)\,c_{q-1}/j!$. Therefore
\[
 \sum_{j=2}^{M}\frac{(-1)^{j+1}}{j}\,\widetilde c_{j-1}
 =\sum_{q=2}^{M}c_{q-1}\sum_{j=q}^{M}\frac{(-1)^{j-1}s(j,q)}{j!}.
\]
Since $s(j,q)=(-1)^{j-q}|s(j,q)|$, the classical identity \cite[Section~6.1]{GrahamKnuthPatashnik}
\[
 \sum_{j=q}^{M}\frac{|s(j,q)|}{j!}=\frac{|s(M+1,q+1)|}{M!}
\]
gives
\begin{align*}
 \sum_{j=2}^{M}\frac{(-1)^{j+1}}{j}\,\widetilde c_{j-1}
 &=\frac1{M!}\sum_{q=2}^{M}(-1)^{q-1}|s(M+1,q+1)|\,c_{q-1}\\
 &=\frac{(-1)^{M-1}}{M!}\sum_{q=2}^{M}s(M+1,q+1)\,c_{q-1},
\end{align*}
which is \eqref{eq:alternating-stirling} after the substitution $j=q+1$.
\end{proof}

We now evaluate the right-hand side of \eqref{eq:alternating-stirling}. For an integer $M\ge2$ and a real number $\lambda$ write $\ff{\lambda}{M+1}=\lambda(\lambda-1)\cdots(\lambda-M)$. The defining property of the Stirling numbers of the first kind, $\sum_{j=1}^{M+1}s(M+1,j)\lambda^j=\ff\lambda{M+1}$, see \cite[Section~6.1]{GrahamKnuthPatashnik}, gives
\begin{equation}\label{eq:stirling-poly}
 \sum_{j=3}^{M+1}s(M+1,j)\,\lambda^{j-2}
 =\frac{\ff\lambda{M+1}}{\lambda^2}-\frac{s(M+1,1)}{\lambda}-s(M+1,2),
 \qquad\lambda\neq0.
\end{equation}
Let $\mathcal B$ be a finite set of positive integers. A \emph{$\mathcal B$-remainder} is a function of $\lambda>0$ and of $M\ge2$ of the form
\[
 R(\lambda)=\lambda^{-2}\sum_{b\in\mathcal B}\beta_b\,\ff{b\lambda}{M+1},
\]
where the real coefficients $\beta_b$ depend neither on $\lambda$ nor on $M$. The sum of a $\mathcal B$-remainder and a $\mathcal B'$-remainder is a $(\mathcal B\cup\mathcal B')$-remainder. If $R$ is a $\mathcal B$-remainder and $c$ is a positive integer, then $\lambda\mapsto R(c\lambda)$ is a $c\mathcal B$-remainder. Since $\ff{b}{M+1}=0$ for all integers $1\le b\le M$, every $\mathcal B$-remainder satisfies
\begin{equation}\label{eq:remainder-vanishes}
 R(1)=0\qquad\text{whenever }\max\mathcal B\le M.
\end{equation}
Note also that $c_m(2)=2^m-1$ by \eqref{eq:c-recurrence}.

\begin{claim}\label{claim:two}
Let $a\ge1$. For all $M\ge2$ and $\lambda>0$,
\[
 \sum_{j=3}^{M+1}s(M+1,j)\,c_{j-2}(2)^a\,\lambda^{j-2}
 =\frac{(-1)^{a-1}s(M+1,1)}{2^a\lambda}+R(\lambda),
\]
where $R$ is a $\mathcal B$-remainder with $\mathcal B=\{1,2,4,\ldots,2^a\}$.
\end{claim}

\begin{proof}
We expand $c_{j-2}(2)^a=(2^{j-2}-1)^a$ by the binomial theorem and apply \eqref{eq:stirling-poly} with $\lambda$ replaced by $2^i\lambda$:
\begin{align*}
 &\sum_{j=3}^{M+1}s(M+1,j)\,c_{j-2}(2)^a\,\lambda^{j-2}
 =\sum_{i=0}^{a}\binom ai(-1)^{a-i}\sum_{j=3}^{M+1}s(M+1,j)\,(2^i\lambda)^{j-2}\\
 &\quad=-s(M+1,2)\sum_{i=0}^{a}\binom ai(-1)^{a-i}
 -\frac{s(M+1,1)}{\lambda}\sum_{i=0}^{a}\binom ai\frac{(-1)^{a-i}}{2^i}\\
 &\qquad+\frac1{\lambda^{2}}\sum_{i=0}^{a}\binom ai\frac{(-1)^{a-i}}{4^i}\ff{2^i\lambda}{M+1}.
\end{align*}
The first sum vanishes because $a\ge1$, the second sum equals $(1/2-1)^a=(-1)^a2^{-a}$, and the last term is a $\mathcal B$-remainder.
\end{proof}

Products of the numbers $c_{j-2}(p)$ with repeated arguments are conveniently encoded by nondecreasing tuples: the product $c_{j-2}(p_1)^{a_1}\cdots c_{j-2}(p_s)^{a_s}$ with $1<p_1<\cdots<p_s$ and positive integers $a_1,\ldots,a_s$ corresponds to the tuple
\[
 \mathbf p=(\underbrace{p_1,\ldots,p_1}_{a_1},\ldots,\underbrace{p_s,\ldots,p_s}_{a_s}).
\]
In this language, Claim~\ref{claim:two} treats the tuples $(2,\ldots,2)$.

We assume that all nondecreasing tuples are ordered colexicographically. Colexicographic order compares tuples from right to left, in contrast to lexicographic order, which compares them from left to right. To compare tuples of different lengths, one extends the shorter tuple by zeros on the right. Equivalently, for two distinct tuples \[
  \mathbf p = (p_1,\ldots,p_N), \quad \mathbf q = (q_1,\ldots,q_M)
\] 
we write $\mathbf q < \mathbf p$ if and only if
\[
  M<N\quad \text{or} \quad \exists m\ \forall i>m\ q_i=p_i\ \text{and}\ q_m < p_m. 
\]

\begin{claim}\label{claim:general}
Let $N\ge1$, and let $\mathbf p=(p_1,\ldots,p_N)$ be a nondecreasing tuple of integers greater than one. Put
\[
 \bar p=p_1+\cdots+p_N,\qquad
 \gamma=(p_1-1)+\cdots+(p_N-1),\qquad
 \Phi=p_1!\cdots p_N!.
\]
Then, for all $M\ge2$ and $\lambda>0$,
\begin{equation}\label{eq:claim-general}
 \sum_{j=3}^{M+1}s(M+1,j)\,c_{j-2}(p_1)\cdots c_{j-2}(p_N)\,\lambda^{j-2}
 =\frac{(-1)^{\gamma+1}s(M+1,1)}{\Phi\,\lambda}+R(\lambda),
\end{equation}
where $R$ is a $\mathcal B$-remainder for some $\mathcal B\subset\{1,\ldots,\bar p!\}$.
\end{claim}

\begin{proof}
We argue by induction along the colexicographic order. The inductive step for $\mathbf p$ described below uses only tuples that are smaller than $\mathbf p$.

If $p_N=2$, then $\mathbf p=(2,\ldots,2)$, and the assertion is Claim~\ref{claim:two} with $a=N$: indeed, $\gamma=N$, $\Phi=2^N$ and $2^N\le(2N)!=\bar p!$. Let now $p=p_N\ge3$, and assume that the claim holds for all tuples smaller than $\mathbf p$. Put
\[
 Q_j=c_{j-2}(p_1)\cdots c_{j-2}(p_{N-1}),
\]
where $Q_j=1$ if $N=1$. By \eqref{eq:c-recurrence} with $m=j-2\ge1$, the left-hand side of \eqref{eq:claim-general} equals
\begin{equation}\label{eq:split}
 \sum_{j=3}^{M+1}s(M+1,j)\,Q_j\,(p!\,\lambda)^{j-2}
 -\sum_{i=1}^{p-1}\sum_{j=3}^{M+1}s(M+1,j)\,Q_j\,c_{j-2}(i)\,\bigl((p-i)!\,\lambda\bigr)^{j-2}.
\end{equation}

The induction rests on the following observation about the lists of block sizes of the sums in \eqref{eq:split}. In the first sum one copy of $p=p_N$ is removed, and the list is $(p_1,\ldots,p_{N-1})$, a proper initial segment of $\mathbf p$. In the $i$th sum the final entry $p_N$ is replaced by $i$, which gives the list $(p_1,\ldots,p_{N-1},i)$; let $\mathbf p^{(i)}$ be the nondecreasing tuple obtained from it by sorting. For $i=1$ the block of size one disappears, because $c_{j-2}(1)=1$, and $\mathbf p^{(1)}=(p_1,\ldots,p_{N-1})$. For $2\le i<p$, let $\ell$ be the greatest index with $p_\ell<p$; then
\[
 \mathbf p^{(i)}=(p_1,\ldots,i,\ldots,p_\ell,\underbrace{p,\ldots,p}_{N-\ell+1}))
\]
agrees with $\mathbf p$ in the last $N-\ell+1$ entries and has the smaller $(\ell+1)$-th entry $p_\ell<p$. Thus every list occurring in \eqref{eq:split} is either empty or a tuple that is colexicographically smaller than $\mathbf p$. Put $\gamma'=\gamma-(p-1)$ and $\Phi'=\Phi/p!$. These are the values of $\gamma$ and $\Phi$ for the tuple $(p_1,\ldots,p_{N-1})$, while for $\mathbf p^{(i)}$ they are $\gamma'+i-1$ and $\Phi'\,i!$, and the entries of $\mathbf p^{(i)}$ sum to at most $\bar p-p+i$.

Suppose first that $N\ge2$. Then all these lists are non-empty, and the induction hypothesis applies to each sum in \eqref{eq:split}, with $\lambda$ replaced by $p!\,\lambda$ and by $(p-i)!\,\lambda$, respectively. Up to remainders, the first sum contributes
\[
 \frac{(-1)^{\gamma'+1}s(M+1,1)}{\Phi'\,p!\,\lambda},
\]
and the $i$th sum contributes
\[
 \frac{(-1)^{\gamma'+i}s(M+1,1)}{\Phi'\,i!\,(p-i)!\,\lambda}.
\]
These remainders are $\mathcal B$-remainders with $\mathcal B\subset\{1,\ldots,\bar p!\}$, because $(\bar p-p)!\,p!\le\bar p!$ and $(\bar p-p+i)!\,(p-i)!\le\bar p!$.

If $N=1$, then $Q_j=1$, and the first sum in \eqref{eq:split} and the sum with $i=1$ are evaluated by \eqref{eq:stirling-poly}. Their terms $-s(M+1,2)$ cancel, their terms containing falling factorials form a $\{(p-1)!,p!\}$-remainder, and their terms containing $s(M+1,1)$ are as in the previous case with $\gamma'=0$ and $\Phi'=1$. The sums with $2\le i\le p-1$ correspond to the one-entry tuples $(i)$ and are handled by the induction hypothesis as before.

In both cases the coefficient of $s(M+1,1)$ equals
\begin{align*}
 \frac{(-1)^{\gamma'+1}}{\Phi'\,p!\,\lambda}
 -\sum_{i=1}^{p-1}\frac{(-1)^{\gamma'+i}}{\Phi'\,i!\,(p-i)!\,\lambda}
 &=\frac{(-1)^{\gamma'+1}}{\Phi\,\lambda}
 \left(1+\sum_{i=1}^{p-1}(-1)^i\binom pi\right)\\
 &=\frac{(-1)^{\gamma'+1}(-1)^{p+1}}{\Phi\,\lambda}
 =\frac{(-1)^{\gamma+1}}{\Phi\,\lambda},
\end{align*}
since $\gamma'+p=\gamma+1$. This proves \eqref{eq:claim-general}.
\end{proof}

\begin{proposition}\label{prop:stirling-sum}
Let $p_1,\ldots,p_s$ be positive integers, not all equal to one, let $r=p_1+\cdots+p_s$, and let $M\ge r!$. Then
\[
 \sum_{j=3}^{M+1}s(M+1,j)\,c_{j-2}(p_1,\ldots,p_s)
 =(-1)^{r-s+M+1}\frac{M!}{p_1!\cdots p_s!}.
\]
\end{proposition}

\begin{proof}
By \eqref{eq:c-multiplicative} and $c_{j-2}(1)=1$, the product $c_{j-2}(p_1)\cdots c_{j-2}(p_s)$ coincides with the product in Claim~\ref{claim:general} for the tuple $\mathbf p$ formed by those of the numbers $p_1,\ldots,p_s$ that exceed one, arranged in nondecreasing order; this tuple is non-empty by assumption. For $\mathbf p$ we have $\bar p\le r$, so that $\bar p!\le r!\le M$, while $\gamma=\sum_{a=1}^s(p_a-1)=r-s$ and $\Phi=p_1!\cdots p_s!$. We apply Claim~\ref{claim:general} with $\lambda=1$. The remainder vanishes by \eqref{eq:remainder-vanishes}, and it remains to note that $s(M+1,1)=(-1)^MM!$.
\end{proof}

\begin{corollary}\label{cor:alternating}
In the setting of Proposition~\ref{prop:stirling-sum},
\[
 \sum_{j=2}^{M}\frac{(-1)^{j+1}}{j}\,\widetilde c_{j-1}(p_1,\ldots,p_s)
 =\frac{(-1)^{r-s}}{p_1!\cdots p_s!}.
\]
\end{corollary}

\begin{proof}
This follows from \eqref{eq:alternating-stirling} and Proposition~\ref{prop:stirling-sum}.
\end{proof}

\subsection{Proof of Lemma~\ref{lem:braid-valuation}}\label{ssec:proof-lemma}

We keep the notation of Section~\ref{ssec:chambers}. All intersections of chambers are non-empty, so the inclusion-exclusion principle \eqref{eq:inclusion-exclusion} gives
\begin{equation}\label{eq:ie-chambers}
 \phi(P)=\sum_{\varnothing\neq I\subset\mathfrak S_r}
 (-1)^{|I|+1}\phi\Bigl(\bigcap_{\sigma\in I}K_\sigma\Bigr).
\end{equation}
Fix a type $[p_1,\ldots,p_s]$ with at least one $p_a>1$. Sets $I$ with $|I|=1$ give chambers. Among the sets $I$ with $|I|=j\ge2$, exactly $(r!/j)\,\widetilde c_{j-1}(p_1,\ldots,p_s)$ have an intersection of this type. Indeed, choose a distinguished chamber $K_\sigma$ in $I$; relabelling the coordinates, we may assume that $\sigma=\mathrm{id}$. The remaining $j-1$ chambers can then be chosen in $\widetilde c_{j-1}(p_1,\ldots,p_s)$ ways, and every set is obtained $j$ times. Since $\widetilde c_{j-1}=0$ for $j>r!$, Corollary~\ref{cor:alternating} with $M=r!$ shows that the total coefficient in \eqref{eq:ie-chambers} of all intersections of type $[p_1,\ldots,p_s]$ equals
\[
 r!\sum_{j=2}^{r!}\frac{(-1)^{j+1}}{j}\,\widetilde c_{j-1}(p_1,\ldots,p_s)
 =(-1)^{r-s}\frac{r!}{p_1!\cdots p_s!}.
\]
Relabelling the coordinates permutes the intersections of a given type transitively, so all $r!/(p_1!\cdots p_s!)$ of them have the same coefficient, which is therefore $(-1)^{r-s}$. The chambers, that is, the intersections of type $[1,\ldots,1]$, occur only for $|I|=1$ and have coefficient $1$, in agreement with the same rule.

Finally, merging consecutive groups of the vectors $y_1,\ldots,y_r$ amounts to merging blocks of $\pi$. Hence the intersections of type $[p_1,\ldots,p_s]$ are exactly the simplices $K_\tau$ with $\tau\in\Ord(\rho)$ for the partitions $\rho\ge\pi$ with $|\rho|=s$. Substituting the coefficients $(-1)^{r-s}=(-1)^{|\pi|-|\rho|}$ into \eqref{eq:ie-chambers} proves \eqref{eq:braid-valuation}.
\qed

\subsection{Partition inversion}\label{ssec:partition-inversion}

The M\"obius function of the partition lattice is given by
\begin{equation}\label{eq:mobius-partition}
 \mu(\pi_0,\rho)=(-1)^{n-|\rho|}\prod_{B\in\rho}(|B|-1)!;
\end{equation}
see \cite[Chapter~3]{Stanley} or \cite{Aigner}.

\begin{proposition}[Permutation-averaging identity]\label{prop:permutation-valuation}
Let $x_1,\ldots,x_n$ be linearly independent, and let $\phi$ be a valuation on convex polytopes. Then
\begin{align}\label{eq:permutation-valuation}
 &\sum_{\sigma\in\mathfrak S_n}
 \phi\bigl(\conv(0,x_{\sigma(1)},x_{\sigma(1)}+x_{\sigma(2)},\ldots,x_1+\cdots+x_n)\bigr)\\
 &\qquad=\sum_{\rho\in\Pi_n}
 \left(\prod_{B\in\rho}(|B|-1)!\right)\phi(P_\rho).\notag
\end{align}
\end{proposition}

\begin{proof}
For $\rho\in\Pi_n$ set
\[
 U(\rho)=\sum_{\tau\in\Ord(\rho)}\phi(K_\tau),
 \qquad G(\rho)=\phi(P_\rho).
\]
Formula \eqref{eq:braid-valuation} is equivalent to
\[
 (-1)^{|\pi|}G(\pi)=\sum_{\rho\ge\pi}(-1)^{|\rho|}U(\rho),\qquad \pi\in\Pi_n.
\]
M\"obius inversion on $\Pi_n$ \cite[Chapter~3]{Stanley} at $\pi_0$ and \eqref{eq:mobius-partition} yield
\[
 U(\pi_0)=\sum_{\rho\in\Pi_n}
 (-1)^{n+|\rho|}\mu(\pi_0,\rho)\,G(\rho)
 =\sum_{\rho\in\Pi_n}
 \left(\prod_{B\in\rho}(|B|-1)!\right)G(\rho).
\]
Since $U(\pi_0)$ is the left-hand side of \eqref{eq:permutation-valuation}, the proof is complete.
\end{proof}

\begin{remark}\label{rem:spitzer}
Exactly $\prod_{B\in\rho}(|B|-1)!$ permutations of $[n]$ have the blocks of $\rho$ as their cycles. Hence the right-hand side of \eqref{eq:permutation-valuation} equals $\sum_{\sigma}\phi(P_{\rho(\sigma)})$, where $\rho(\sigma)$ is the partition of $[n]$ into the cycles of $\sigma$. Identities of this type go back to the combinatorial lemmas of Spitzer \cite{Spitzer}, Baxter \cite{Baxter} and Barndorff-Nielsen and Baxter \cite{BarndorffBaxter}; see also \cite{Steele} and \cite[Section~5]{VysotskyZaporozhets}. For example, for a fixed vector $u$ the support function $\phi(K)=h(K,u)$ is a valuation, and \eqref{eq:permutation-valuation} becomes
\[
 \sum_{\sigma\in\mathfrak S_n}\max_{0\le j\le n}\ip{x_{\sigma(1)}+\cdots+x_{\sigma(j)}}{u}
 =\sum_{\sigma\in\mathfrak S_n}\sum_{c}\max\{0,\ip{x_c}u\},
\]
where $c$ runs over the cycles of $\sigma$ and $x_c=\sum_{i\in c}x_i$. This is the summed form of Spitzer's combinatorial lemma \cite{Spitzer}.
\end{remark}

We also need a weighted partition identity. It is classical: by the first sentence of Remark~\ref{rem:spitzer}, its left-hand side counts the permutations of an $N$-element set according to their cycles. We include the inductive proof from \cite[Section~7]{BolotinThesis} to keep the exposition self-contained. We use the convention that $\Pi_0$ consists of the empty partition, whose weight is the empty product.

\begin{lemma}\label{lem:partition-weight}
For every $N\ge0$,
\begin{equation}\label{eq:partition-weight}
 \sum_{\rho\in\Pi_N}\prod_{B\in\rho}(|B|-1)!=N!.
\end{equation}
\end{lemma}

\begin{proof}
For $N=0$ the only partition is the empty one, whose weight is $1=0!$, and the case $N=1$ is immediate. Let $N\ge2$, and suppose that the identity holds for all smaller sets. Fix an element $x\in[N]$. For $1\le j\le N$, let $\Pi'_{N,j}$ be the set of partitions in which the block containing $x$ has size $j$; these sets partition $\Pi_N$. The other $j-1$ elements of the block containing $x$ can be chosen in $\binom{N-1}{j-1}$ ways, this block contributes the factor $(j-1)!$, and the remaining $N-j$ elements form an arbitrary partition of a set of size $N-j$. By the induction hypothesis,
\[
 \sum_{\rho\in\Pi'_{N,j}}\prod_{B\in\rho}(|B|-1)!
 =\binom{N-1}{j-1}(j-1)!\,(N-j)!
 =(N-1)!.
\]
Summing over $j=1,\ldots,N$, we obtain \eqref{eq:partition-weight}.
\end{proof}

%======================================================================
\section{Proof of the discrete theorems}\label{sec:discrete-proof}

All identities in this section are between non-negative quantities and are understood in $[0,\infty]$.

\subsection{A sufficient condition for finiteness}

Suppose that $m_i\le n_i$ for all $i$; otherwise both sides of \eqref{eq:exchangeable-main} vanish. Since $S_{j,i}\in\sum_{\ell=1}^{n_i}\seg{\xi_{\ell,i}}$ for all $j$, each hull $C_i$ is contained in the zonotope $\sum_{\ell=1}^{n_i}\seg{\xi_{\ell,i}}$. By monotonicity and multilinearity, $V_d(C_1[m_1],\ldots,C_k[m_k])$ is bounded by the sum of the mixed volumes of the segments $\seg{\xi_{\ell_{r,i},i}}$, $1\le r\le m_i$, $1\le i\le k$, over all choices of indices $\ell_{r,i}\in\{1,\ldots,n_i\}$. Terms in which an increment appears twice vanish, because two of the segments are then parallel. The remaining terms are at most $\frac1{d!}\prod_{i}\prod_r\|\xi_{\ell_{r,i},i}\|$ by \eqref{eq:segments-determinant} and Hadamard's inequality. Taking expectations and using exchangeability within the sequences and independence between them, we obtain
\begin{equation}\label{eq:finiteness}
 \E V_d\bigl(C_1[m_1],\ldots,C_k[m_k]\bigr)
 \le\frac1{d!}\prod_{i=1}^k\frac{n_i!}{(n_i-m_i)!}\,
 \E\bigl(\|\xi_{1,i}\|\cdots\|\xi_{m_i,i}\|\bigr).
\end{equation}
In particular, both sides of \eqref{eq:exchangeable-main} are finite if $\E(\|\xi_{1,i}\|\cdots\|\xi_{m_i,i}\|)<\infty$ for every $i$, for example, if the increments are independent and integrable.

\subsection{Replacing one hull by block segments}

\begin{proposition}\label{prop:one-walk}
Let $\xi_1,\ldots,\xi_n$ be a finite exchangeable sequence of random vectors in $\R^q$ that are linearly independent almost surely; in particular, $q\ge n$. Put $S_j=\xi_1+\cdots+\xi_j$ and $C=\conv(0,S_1,\ldots,S_n)$. Let $1\le m\le q$, and let $L_1,\ldots,L_{q-m}$ be random compact convex sets independent of $(\xi_1,\ldots,\xi_n)$. For $\mathbf j=(j_1,\ldots,j_m)\in\J_m(n)$ set $a_0=0$, $a_r=j_1+\cdots+j_r$, and
\[
 Y_r(\mathbf j)=\sum_{\ell=a_{r-1}+1}^{a_r}\xi_\ell,
 \qquad r=1,\ldots,m.
\]
Then
\begin{align}\label{eq:one-walk}
 &\E V_q\bigl(L_1,\ldots,L_{q-m},C[m]\bigr)\\
 &\quad=\sum_{\mathbf j\in\J_m(n)}
 \frac{\E V_q\bigl(L_1,\ldots,L_{q-m},
 \seg{Y_1(\mathbf j)},\ldots,\seg{Y_m(\mathbf j)}\bigr)}{j_1\cdots j_m}.\notag
\end{align}
\end{proposition}

\begin{proof}
For fixed realizations of $L_1,\ldots,L_{q-m}$ and of the increments, the functional
\[
 \phi(K)=V_q\bigl(L_1,\ldots,L_{q-m},K[m]\bigr)
\]
is a valuation by \eqref{eq:diagonal-valuation}. Apply Proposition~\ref{prop:permutation-valuation} with $x_j=\xi_j$ and take expectations. By exchangeability, every path simplex on the left-hand side of \eqref{eq:permutation-valuation} has, jointly with $(L_1,\ldots,L_{q-m})$, the same distribution as $C$. Hence
\begin{equation}\label{eq:one-walk-partitions}
 n!\,\E\phi(C)
 =\sum_{\rho\in\Pi_n}w(\rho)\,\E\phi(P_\rho),
 \qquad
 w(\rho)=\prod_{B\in\rho}(|B|-1)!.
\end{equation}
By Minkowski multilinearity,
\begin{equation}\label{eq:expand-zonotope}
 \phi(P_\rho)
 =\sum_{(B_1,\ldots,B_m)}
 V_q\bigl(L_1,\ldots,L_{q-m},
 \seg{\xi_{B_1}},\ldots,\seg{\xi_{B_m}}\bigr),
\end{equation}
where $\xi_B=\sum_{\ell\in B}\xi_\ell$ and the sum runs over the ordered $m$-tuples of distinct blocks of $\rho$; terms with a repeated block vanish because two of the segments would coincide.

Fix positive integers $j_1,\ldots,j_m$ with $J=j_1+\cdots+j_m\le n$. There are $n!/(j_1!\cdots j_m!\,(n-J)!)$ ordered $m$-tuples $(B_1,\ldots,B_m)$ of disjoint subsets of $[n]$ with $|B_r|=j_r$. For each of them, by Lemma~\ref{lem:partition-weight}, the partitions $\rho$ containing $B_1,\ldots,B_m$ as blocks have total weight
\[
 \prod_{r=1}^m(j_r-1)!\sum_{\rho'\in\Pi_{n-J}}w(\rho')=\prod_{r=1}^m(j_r-1)!\,(n-J)!.
\]
By exchangeability, the vector $(\xi_{B_1},\ldots,\xi_{B_m})$ has the same distribution as $(Y_1(\mathbf j),\ldots,Y_m(\mathbf j))$, and both vectors are independent of $L_1,\ldots,L_{q-m}$. Consequently, the terms with block sizes $j_1,\ldots,j_m$ contribute to the right-hand side of \eqref{eq:one-walk-partitions}
\begin{multline*}
 \frac{n!}{j_1!\cdots j_m!\,(n-J)!}\prod_{r=1}^m(j_r-1)!\,(n-J)!\\
 \times\E V_q\bigl(L_1,\ldots,L_{q-m},\seg{Y_1(\mathbf j)},\ldots,\seg{Y_m(\mathbf j)}\bigr),
\end{multline*}
that is, $n!/(j_1\cdots j_m)$ times this expectation. Summing over $\mathbf j\in\J_m(n)$ and dividing by $n!$, we obtain \eqref{eq:one-walk}.
\end{proof}

\subsection{Proof of Theorem~\ref{thm:exchangeable} in general position}

Suppose first that, for every $i$, the increments $\xi_{1,i},\ldots,\xi_{n_i,i}$ are linearly independent almost surely; in particular, $n_i\le d$. We may assume that $m_i\ge1$ for all $i$, since families with $m_i=0$ do not enter \eqref{eq:exchangeable-main}. Apply Proposition~\ref{prop:one-walk} with $q=d$ to the hull $C_k$ taken with multiplicity $m_k$, the sets $L$ being $C_1[m_1],\ldots,C_{k-1}[m_{k-1}]$; they are independent of $\boldsymbol\xi_k$. This replaces $C_k[m_k]$ by the segments $\seg{Y^{(1,k)}_{\mathbf j_k}},\ldots,\seg{Y^{(m_k,k)}_{\mathbf j_k}}$. Next apply Proposition~\ref{prop:one-walk} to $C_{k-1}[m_{k-1}]$. Now the sets $L$ are $C_1,\ldots,C_{k-2}$, taken with their multiplicities, and the segments built from $\boldsymbol\xi_k$; all of them are independent of $\boldsymbol\xi_{k-1}$. Continuing in this way, we obtain
\begin{align}\label{eq:all-segments}
 &\E V_d\bigl(C_1[m_1],\ldots,C_k[m_k]\bigr)\\
 &\quad=
 \sum_{\substack{\mathbf j_i\in\J_{m_i}(n_i)\\i=1,\ldots,k}}
 \frac{
 \E V_d\bigl(
 \seg{Y_{\mathbf j_1}^{(1,1)}},\ldots,
 \seg{Y_{\mathbf j_k}^{(m_k,k)}}
 \bigr)}
 {\displaystyle\prod_{i=1}^k\prod_{r=1}^{m_i}j_{r,i}}.\notag
\end{align}
Formula \eqref{eq:segments-determinant} turns \eqref{eq:all-segments} into \eqref{eq:exchangeable-main}.

\subsection{Removing the genericity assumption}\label{ssec:lifting}

It remains to remove the assumption that the increments are linearly independent almost surely; note that this assumption also forces $n_i\le d$. We use a Gaussian lifting. Let
\[
 n_*=\max_{1\le i\le k}n_i,\qquad q=d+n_*,
\]
and identify $\R^q$ with $\R^d\times\R^{n_*}$. For every increment $\xi_{r,i}$ take a standard Gaussian vector $g_{r,i}$ in $\R^{n_*}$, all these vectors being independent of each other and of the increments, and put
\[
 \widetilde\xi_{r,i}=(\xi_{r,i},g_{r,i})\in\R^q.
\]
For each $i$, the sequence $(\widetilde\xi_{1,i},\ldots,\widetilde\xi_{n_i,i})$ is exchangeable, and its vectors are linearly independent almost surely because $n_i\le n_*$. The $k$ lifted sequences are independent. Let $\widetilde C_i$ be the convex hull of the partial sums of the $i$th lifted sequence, so that $\widetilde C_i|E=C_i$, where $E=\R^d\times\{0\}$. Finally, let $f_1,\ldots,f_{n_*}$ be the standard basis of $\{0\}\times\R^{n_*}$, and regard each segment $\seg{f_s}$ as the convex hull of a deterministic walk with one step.

Apply the formula proved in the previous subsection, in dimension $q$, to
\[
 \widetilde C_1[m_1],\ldots,\widetilde C_k[m_k],
 \seg{f_1},\ldots,\seg{f_{n_*}}.
\]
By \eqref{eq:iterated-projection}, the left-hand side equals
\begin{equation}\label{eq:lift-left}
 \frac{d!}{q!}\,
 \E V_d\bigl(C_1[m_1],\ldots,C_k[m_k]\bigr).
\end{equation}
On the right-hand side, each one-step walk contributes only the index $j=1$, and the block triangular structure of the matrix gives
\begin{equation}\label{eq:lift-det}
 \det\begin{bmatrix}
 \widetilde Y_{\mathbf j_1}^{(1,1)}&\cdots&\widetilde Y_{\mathbf j_k}^{(m_k,k)}&f_1&\cdots&f_{n_*}
 \end{bmatrix}
 =\det\begin{bmatrix}
 Y_{\mathbf j_1}^{(1,1)}&\cdots&Y_{\mathbf j_k}^{(m_k,k)}
 \end{bmatrix},
\end{equation}
where $\widetilde Y$ denotes the block sums of the lifted sequences. The formula in dimension $q$ carries the prefactor $1/q!$. Comparing it with \eqref{eq:lift-left} and using \eqref{eq:lift-det}, we obtain \eqref{eq:exchangeable-main} with the prefactor $1/d!$. This proves Theorem~\ref{thm:exchangeable} in full generality.

\subsection{Independent increments}

If the increments within each sequence are i.i.d., then the block sums in \eqref{eq:block-sums} are independent and satisfy
\[
 Y_{\mathbf j_i}^{(r,i)}\stackrel d=S^{(r,i)}_{j_{r,i}}.
\]
This proves Corollary~\ref{thm:discrete}. The finiteness of all quantities in \eqref{eq:discrete-main} follows from \eqref{eq:finiteness}, or directly from Hadamard's inequality and the finite first moments of the mutually independent columns. Corollary~\ref{cor:discrete-minkowski} follows from Corollary~\ref{thm:discrete} and \eqref{eq:minkowski-polynomial}.

\subsection{Intrinsic volumes}\label{ssec:intrinsic}

Proposition~\ref{prop:one-walk} with balls in place of the sets $L_1,\ldots,L_{q-m}$ yields the known formulas for the expected intrinsic volumes of convex hulls of random walks. Recall that $B^d$ is the Euclidean unit ball in $\R^d$ and $\kappa_d=\Vol_d(B^d)$. The intrinsic volumes $\IV_0(K),\ldots,\IV_d(K)$ of a set $K\in\mathcal K^d$ are defined by the Steiner formula
\[
 \Vol_d(K+\lambda B^d)=\sum_{m=0}^{d}\kappa_{d-m}\,\IV_m(K)\,\lambda^{d-m},
 \qquad\lambda\ge0;
\]
see \cite[Chapter~4]{Schneider}. They do not depend on the dimension of the ambient space, and $\IV_m(K)$ is the $m$-dimensional volume of $K$ if $K$ is contained in an $m$-dimensional affine subspace. Comparing the Steiner formula with \eqref{eq:minkowski-polynomial} for $K_1=K$ and $K_2=B^d$, we obtain
\begin{equation}\label{eq:intrinsic-mixed}
 V_d\bigl(K[m],B^d[d-m]\bigr)=\frac{\kappa_{d-m}}{\binom dm}\,\IV_m(K),
 \qquad K\in\mathcal K^d,\ 0\le m\le d.
\end{equation}
For $y_1,\ldots,y_m\in\R^d$ let
\[
 \|y_1\wedge\cdots\wedge y_m\|=\bigl(\det[\ip{y_a}{y_b}]_{a,b=1}^m\bigr)^{1/2}
\]
be the $m$-dimensional volume of the parallelotope $P=\seg{y_1}+\cdots+\seg{y_m}$; for $m=d$ it equals $|\det[y_1,\ldots,y_d]|$. If we expand $V_d(P[m],B^d[d-m])$ by multilinearity, the terms containing one of the segments twice vanish, and the remaining $m!$ terms coincide by symmetry. Since $\IV_m(P)=\|y_1\wedge\cdots\wedge y_m\|$, formula \eqref{eq:intrinsic-mixed} gives
\begin{equation}\label{eq:segments-ball}
 V_d\bigl(\seg{y_1},\ldots,\seg{y_m},B^d[d-m]\bigr)
 =\frac{\kappa_{d-m}}{m!\binom dm}\,\|y_1\wedge\cdots\wedge y_m\|.
\end{equation}

\begin{corollary}[Expected intrinsic volumes]\label{cor:intrinsic}
Let $\xi_1,\ldots,\xi_n$ be a finite exchangeable sequence of random vectors in $\R^d$, let $S_j=\xi_1+\cdots+\xi_j$ and $C=\conv(0,S_1,\ldots,S_n)$, and define the block sums $Y_r(\mathbf j)$ as in Proposition~\ref{prop:one-walk}. Then, for $1\le m\le d$,
\begin{equation}\label{eq:intrinsic-exchangeable}
 \E\IV_m(C)=\frac1{m!}\sum_{\mathbf j\in\J_m(n)}
 \frac{\E\|Y_1(\mathbf j)\wedge\cdots\wedge Y_m(\mathbf j)\|}{j_1\cdots j_m}.
\end{equation}
If, in addition, the increments are i.i.d.\ and integrable, then all quantities are finite and
\begin{equation}\label{eq:intrinsic-iid}
 \E\IV_m(C)=\frac1{m!}\sum_{\mathbf j\in\J_m(n)}
 \frac{\E\|S^{(1)}_{j_1}\wedge\cdots\wedge S^{(m)}_{j_m}\|}{j_1\cdots j_m},
\end{equation}
where $S^{(1)},\ldots,S^{(m)}$ are independent copies of the walk.
\end{corollary}

\begin{proof}
If the increments are linearly independent almost surely, it suffices to take $q=d$ and $L_1=\cdots=L_{d-m}=B^d$ in Proposition~\ref{prop:one-walk} and to apply \eqref{eq:intrinsic-mixed} and \eqref{eq:segments-ball}. In general, we combine this choice with the lifting of Section~\ref{ssec:lifting}. Let $q=d+n$, and put $\widetilde\xi_\ell=(\xi_\ell,g_\ell)\in\R^q$, where $g_1,\ldots,g_n$ are independent standard Gaussian vectors in $\R^n$, independent of the increments. The sequence $(\widetilde\xi_1,\ldots,\widetilde\xi_n)$ is exchangeable, and its vectors are linearly independent almost surely. Let $\widetilde C$ be the convex hull of its partial sums, let $E=\R^d\times\{0\}$, and let $f_1,\ldots,f_n$ be the standard basis of $\{0\}\times\R^n$. We apply Proposition~\ref{prop:one-walk} in $\R^q$ to $\widetilde C$, the $q-m$ sets $L$ being $d-m$ copies of the unit ball $B^q$ and the segments $\seg{f_1},\ldots,\seg{f_n}$. We have $\widetilde C|E=C$ and $B^q|E=B^d$, and the segments spanned by the block sums of the lifted sequence project onto the segments $\seg{Y_r(\mathbf j)}$. Therefore, by \eqref{eq:iterated-projection}, both sides of \eqref{eq:one-walk} are equal to $d!/q!$ times the corresponding quantities in $\R^d$, and we obtain
\[
 \E V_d\bigl(C[m],B^d[d-m]\bigr)
 =\sum_{\mathbf j\in\J_m(n)}
 \frac{\E V_d\bigl(\seg{Y_1(\mathbf j)},\ldots,\seg{Y_m(\mathbf j)},B^d[d-m]\bigr)}{j_1\cdots j_m}.
\]
Now \eqref{eq:intrinsic-exchangeable} follows from \eqref{eq:intrinsic-mixed} and \eqref{eq:segments-ball}. For i.i.d.\ increments the block sums $Y_1(\mathbf j),\ldots,Y_m(\mathbf j)$ are independent and $Y_r(\mathbf j)\stackrel d=S_{j_r}$, which gives \eqref{eq:intrinsic-iid}; its right-hand side is finite by Hadamard's inequality.
\end{proof}

For i.i.d.\ increments whose distribution puts no mass on affine hyperplanes, \eqref{eq:intrinsic-iid} is the formula of Vysotsky and Zaporozhets \cite[Corollary~3]{VysotskyZaporozhets}; Corollary~\ref{cor:intrinsic} requires no such assumption. For $m=1$ both \eqref{eq:intrinsic-exchangeable} and \eqref{eq:intrinsic-iid} become
\[
 \E\IV_1(C)=\sum_{j=1}^n\frac{\E\|S_j\|}{j},
\]
the multidimensional version of the Spitzer-Widom formula obtained in \cite{VysotskyZaporozhets}. In the plane $2\IV_1(C)$ is the perimeter of $C$, and we recover the formula of Spitzer and Widom \cite{SpitzerWidom}. For $m=d$, formula \eqref{eq:intrinsic-exchangeable} is the case $k=1$ of Theorem~\ref{thm:exchangeable}.

%======================================================================
\section{Proof of the continuous-time results}\label{sec:continuous-proof}

We follow the approximation scheme of Molchanov and Wespi \cite[proof of Theorem~2.3]{MolchanovWespi}: the processes are sampled on finite grids, the random-walk formula of Corollary~\ref{thm:discrete} is applied to the sampled processes, and one passes to the limit. Since the processes are not assumed to be self-similar, the limits of the resulting lattice sums are identified by means of improper Riemann sums. We begin with three analytic lemmas.

\begin{lemma}[Small-time first moment]\label{lem:small-time}
Let $X$ be a L\'evy process in $\R^d$ with $\E\|X(1)\|<\infty$. For every $T>0$ there is a finite constant $c_T$ such that
\begin{equation}\label{eq:small-time}
 \E\|X(t)\|\le c_T\sqrt t,
 \qquad 0<t\le T.
\end{equation}
Moreover, $X$ is continuous in $L^1$:
\begin{equation}\label{eq:L1-continuity}
 \E\|X(t)-X(s)\|\longrightarrow0
 \qquad\text{as }t\to s.
\end{equation}
\end{lemma}

\begin{proof}
By the L\'evy-It\^o decomposition \cite[Theorem~19.2]{Sato},
\[
 X(t)=bt+QB(t)+M(t)+J(t),
\]
where $b\in\R^d$, $Q$ is a $d\times d$ matrix, $B$ is a standard Brownian motion in $\R^d$, $M$ is the compensated sum of the jumps of norm at most one, and $J(t)$ is the sum of the jumps of norm greater than one up to time $t$. Let $\nu$ be the L\'evy measure of $X$. The assumption $\E\|X(1)\|<\infty$ is equivalent to $\int_{\|x\|>1}\|x\|\,\nu(\dd x)<\infty$ \cite[Theorem~25.3]{Sato}. Consequently,
\begin{align*}
 \E\|bt\|&=\|b\|\,t,\\
 \E\|QB(t)\|&\le\bigl(\E\|QB(t)\|^2\bigr)^{1/2}=\|Q\|_{\mathrm{HS}}\sqrt t,\\
 \E\|M(t)\|&\le\bigl(\E\|M(t)\|^2\bigr)^{1/2}
 =\left(t\int_{\|x\|\le1}\|x\|^2\nu(\dd x)\right)^{1/2},\\
 \E\|J(t)\|&\le t\int_{\|x\|>1}\|x\|\,\nu(\dd x),
\end{align*}
where $\|Q\|_{\mathrm{HS}}$ is the Hilbert-Schmidt norm of $Q$. On a bounded interval the terms of order $t$ are bounded by constants times $\sqrt t$, which proves \eqref{eq:small-time}. Since $X(t)-X(s)$ has the same distribution as $X(|t-s|)$ or $-X(|t-s|)$, \eqref{eq:small-time} also gives \eqref{eq:L1-continuity}.
\end{proof}

For an array of times $\mathbf t=(t_{r,i})\in(0,\infty)^d$ define
\begin{equation}\label{eq:integrand-F}
 F(\mathbf t)=
 \frac{\E\mathcal D_X(\mathbf t)}
 {\displaystyle\prod_{i=1}^k\prod_{r=1}^{m_i}t_{r,i}}.
\end{equation}

\begin{lemma}[Continuity and domination]\label{lem:integrand}
The function $F$ is continuous on $(0,\infty)^d$. On the set $\prod_i\Delta_{m_i}(T_i)$ it satisfies
\begin{equation}\label{eq:integrand-bound}
 F(\mathbf t)
 \le c\prod_{i=1}^k\prod_{r=1}^{m_i}t_{r,i}^{-1/2}
\end{equation}
for a finite constant $c$ depending only on the processes and the horizons. In particular, $F$ is integrable over $\prod_i\Delta_{m_i}(T_i)$.
\end{lemma}

\begin{proof}
Hadamard's inequality and the independence of all copies imply
\[
 \E\mathcal D_X(\mathbf t)
 \le\prod_{i=1}^k\prod_{r=1}^{m_i}\E\|X_i(t_{r,i})\|,
\]
and Lemma~\ref{lem:small-time} gives \eqref{eq:integrand-bound}. The right-hand side of \eqref{eq:integrand-bound} is integrable, as follows by iterated integration or from the Dirichlet integral \eqref{eq:dirichlet-stable} below.

For the continuity, replace the columns one at a time and use the inequality
\begin{equation}\label{eq:column-replacement}
 \bigl||\det[x_1,\ldots,x_d]|-|\det[y_1,\ldots,y_d]|\bigr|
 \le\sum_{r=1}^d\|x_r-y_r\|
 \prod_{s<r}\|x_s\|\prod_{s>r}\|y_s\|,
\end{equation}
which follows from the multilinearity of the determinant and Hadamard's inequality. Taking expectations and using independence and the local boundedness of $t\mapsto\E\|X_i(t)\|$, we see from \eqref{eq:L1-continuity} that the numerator in \eqref{eq:integrand-F} is continuous, and hence so is $F$.
\end{proof}

\begin{lemma}[Improper Riemann sums]\label{lem:riemann}
Let $n\to\infty$ and put $h_i=T_i/n$. Then
\begin{align}\label{eq:riemann-sum}
 &\left(\prod_{i=1}^k h_i^{m_i}\right)
 \sum_{\substack{\mathbf j_i\in\J_{m_i}(n)\\i=1,\ldots,k}}
 F\bigl((j_{r,i}h_i)_{r,i}\bigr)
 \longrightarrow
 \int_{\prod_{i=1}^k\Delta_{m_i}(T_i)}F(\mathbf t)\,\dd\mathbf t.
\end{align}
\end{lemma}

\begin{proof}
The left-hand side of \eqref{eq:riemann-sum} is a Riemann sum of $F$ over the lattice cells $\prod_{i,r}\bigl((j_{r,i}-1)h_i,\,j_{r,i}h_i\bigr]$, with $F$ evaluated at the upper corners of the cells; the condition $\mathbf j_i\in\J_{m_i}(n)$ means precisely that these corners lie in $\prod_i\Delta_{m_i}(T_i)$. Fix $\delta>0$. On the compact part of the domain where all coordinates are at least $\delta$, the function $F$ is continuous by Lemma~\ref{lem:integrand}, and the usual convergence of Riemann sums applies. It remains to control the region where at least one coordinate is less than $\delta$. The integral over this region tends to zero as $\delta\downarrow0$ by \eqref{eq:integrand-bound} and dominated convergence. For the lattice sums we use that, uniformly in $h>0$,
\[
 h\sum_{1\le j\le\delta/h}(jh)^{-1/2}\le 2\sqrt\delta,
 \qquad
 h\sum_{1\le j\le T/h}(jh)^{-1/2}\le 2\sqrt T.
\]
Since the evaluation points are the upper corners, at which every coordinate is positive, a union bound over the $d$ coordinates and \eqref{eq:integrand-bound} show that the contribution of the lattice points with some coordinate less than $\delta$ is $O(\sqrt\delta)$ uniformly in $n$. Letting first $n\to\infty$ and then $\delta\downarrow0$, we obtain \eqref{eq:riemann-sum}.
\end{proof}

We also need the pathwise convergence of the sampled hulls.

\begin{lemma}[Dyadic hulls]\label{lem:dyadic}
Let $X$ be a c\`adl\`ag process in $\R^d$ and fix $T>0$. For $\ell\ge0$ let
\[
 D_\ell(T)=\{jT2^{-\ell}:j=0,1,\ldots,2^\ell\},
 \qquad
 Z^{(\ell)}(T)=\conv\{X(t):t\in D_\ell(T)\}.
\]
Then $Z^{(\ell)}(T)\uparrow Z(T)=\clconv\{X(t):0\le t\le T\}$ as $\ell\to\infty$, and the convergence holds in the Hausdorff metric.
\end{lemma}

\begin{proof}
A c\`adl\`ag path on $[0,T]$ is bounded. For every $t<T$ there are dyadic points $t_\ell\downarrow t$, and right continuity gives $X(t_\ell)\to X(t)$; the endpoint $T$ belongs to every grid. Hence the closure of the union of the increasing sets $Z^{(\ell)}(T)$ is $Z(T)$. An increasing sequence of compact sets whose union is dense in a compact set converges to it in the Hausdorff metric.
\end{proof}

\begin{proof}[Proof of Theorem~\ref{thm:continuous}]
Fix a dyadic level $\ell$, and put $n=2^\ell$ and $h_i=T_i/n$. The sampled process
\[
 S^{(\ell)}_{j,i}=X_i(jh_i),\qquad j=0,\ldots,n,
\]
is a random walk with i.i.d.\ integrable increments, the walks corresponding to different $i$ are independent, and the convex hull of the $i$th walk is
\[
 C_i^{(\ell)}=\conv\{X_i(jh_i):j=0,\ldots,n\}.
\]
By Corollary~\ref{thm:discrete},
\begin{align*}
 \E V_d\bigl(C_1^{(\ell)}[m_1],\ldots,C_k^{(\ell)}[m_k]\bigr)
 &=\frac1{d!}
 \sum_{\substack{\mathbf j_i\in\J_{m_i}(n)\\i=1,\ldots,k}}
 \frac{\E\mathcal D_X\bigl((j_{r,i}h_i)_{r,i}\bigr)}
 {\displaystyle\prod_{i=1}^k\prod_{r=1}^{m_i}j_{r,i}}\\
 &=\frac1{d!}\left(\prod_{i=1}^k h_i^{m_i}\right)
 \sum_{\substack{\mathbf j_i\in\J_{m_i}(n)\\i=1,\ldots,k}}
 F\bigl((j_{r,i}h_i)_{r,i}\bigr).
\end{align*}
By Lemma~\ref{lem:riemann}, the right-hand side converges to the right-hand side of \eqref{eq:continuous-main} as $\ell\to\infty$.

On the other hand, $C_i^{(\ell)}\uparrow Z_i(T_i)$ in the Hausdorff metric almost surely by Lemma~\ref{lem:dyadic}. Mixed volumes are monotone under inclusion and continuous \cite[Section~5.1]{Schneider}, so that
\[
 V_d\bigl(C_1^{(\ell)}[m_1],\ldots,C_k^{(\ell)}[m_k]\bigr)
 \uparrow
 V_d\bigl(Z_1(T_1)[m_1],\ldots,Z_k(T_k)[m_k]\bigr)
\]
almost surely. The monotone convergence theorem completes the proof of \eqref{eq:continuous-main}. The integral is finite by Lemma~\ref{lem:integrand}.
\end{proof}

\begin{proof}[Proof of Corollary~\ref{cor:exponential}]
Condition on $(\tau_1,\ldots,\tau_k)$ and insert \eqref{eq:continuous-main}. The integrand is non-negative, so Tonelli's theorem applies, and for every $i$
\[
 \int_{t_{1,i}+\cdots+t_{m_i,i}}^\infty q_i e^{-q_iT}\,\dd T
 =e^{-q_i(t_{1,i}+\cdots+t_{m_i,i})}.
\]
Multiplying these factors gives \eqref{eq:exponential-main}. To see that the right-hand side is finite, note that the function $t\mapsto\E\|X_i(t)\|$ is subadditive by the independence and stationarity of increments. Together with \eqref{eq:small-time} for $T=1$ this gives $\E\|X_i(t)\|\le c\,(\sqrt t+t)$ for all $t>0$ and some finite constant $c$. By Hadamard's inequality, $F(\mathbf t)\le c^d\prod_{i,r}\bigl(t_{r,i}^{-1/2}+1\bigr)$, which is integrable against the exponential weight in \eqref{eq:exponential-main}.
\end{proof}

\begin{proof}[Proof of Corollary~\ref{cor:stable} and Examples~\ref{ex:brownian} and~\ref{ex:isotropic}]
Self-similarity gives $X_i(t)\stackrel d=t^{1/\alpha_i}X_i(1)$. Hence, by \eqref{eq:random-segments} and the independence of the columns,
\[
 \E\mathcal D_X(\mathbf t)
 =d!\,V_d\bigl(H_1[m_1],\ldots,H_k[m_k]\bigr)
 \prod_{i=1}^k\prod_{r=1}^{m_i}t_{r,i}^{1/\alpha_i},
 \qquad H_i=\E\seg{X_i(1)}.
\]
Therefore the integral in Theorem~\ref{thm:continuous} separates into $k$ Dirichlet integrals
\begin{equation}\label{eq:dirichlet-stable}
 \int_{\Delta_m(T)}\prod_{r=1}^m t_r^{1/\alpha-1}\,\dd t_1\cdots\dd t_m
 =T^{m/\alpha}\frac{\Gamma(1/\alpha)^m}{\Gamma(1+m/\alpha)}.
\end{equation}
Finally, \eqref{eq:stable-selection} and the homogeneity of mixed volumes give
\[
 V_d\bigl(H_1[m_1],\ldots,H_k[m_k]\bigr)
 =\pi^{-d}\prod_{i=1}^k\Gamma(1-1/\alpha_i)^{m_i}\;V_d\bigl(K_1[m_1],\ldots,K_k[m_k]\bigr),
\]
which yields \eqref{eq:stable-main}. Example~\ref{ex:brownian} follows from $\Gamma(1/2)=\sqrt\pi$ and $K_i=2^{-1/2}B^d$, and Example~\ref{ex:isotropic} follows from $K_i=c_i^{1/\alpha_i}B^d$ and $V_d(B^d[d])=\kappa_d$.
\end{proof}

%======================================================================
\section{Random projections of orthoschemes}\label{sec:orthoschemes}

For $K\in\mathcal K^n$ we denote by $\IV_d(K)$ its $d$th intrinsic volume; see Section~\ref{ssec:intrinsic}. Let $A$ be a random $d\times n$ matrix with independent standard Gaussian entries; the set $AK\subset\R^d$ is sometimes called the $d$-dimensional Gaussian spectrum of $K$. By the theorem of Sudakov and Tsirelson \cite{Sudakov,Tsirelson}, in the finite-dimensional form given in \cite[Section~4.1]{VysotskyZaporozhets},
\begin{equation}\label{eq:tsirelson}
 \IV_d(K)=\frac{(2\pi)^{d/2}}{d!\,\kappa_d}\,
 \E\Vol_d(AK),
\end{equation}
while Kubota's formula \cite{SchneiderWeil} gives
\begin{equation}\label{eq:kubota}
 \IV_d(K)=\binom nd\frac{\kappa_n}{\kappa_d\kappa_{n-d}}
 \E_L\Vol_d(K|L),
\end{equation}
where $L$ is a Haar-distributed random $d$-dimensional linear subspace of $\R^n$.

\begin{proof}[Proof of Theorem~\ref{thm:orthoscheme}]
Apply \eqref{eq:tsirelson} and \eqref{eq:kubota} to
\[
 W=\lambda_1\Omega_1+\cdots+\lambda_k\Omega_k,\qquad\lambda_1,\ldots,\lambda_k\ge0.
\]
Since both $A$ and the projection onto $L$ are linear,
\[
 AW=\lambda_1A\Omega_1+\cdots+\lambda_kA\Omega_k,
 \qquad
 W|L=\lambda_1(\Omega_1|L)+\cdots+\lambda_k(\Omega_k|L).
\]
Expanding $\Vol_d(AW)$ and $\Vol_d(W|L)$ by \eqref{eq:minkowski-polynomial} and comparing the coefficients of $\lambda_1^{m_1}\cdots\lambda_k^{m_k}$, we find
\begin{align}\label{eq:spectrum-projection}
 &\E_L V_d\bigl((\Omega_1|L)[m_1],\ldots,(\Omega_k|L)[m_k]\bigr)\\
 &\quad=\frac{(2\pi)^{d/2}(n-d)!\,\kappa_{n-d}}{n!\,\kappa_n}
 \E V_d\bigl((A\Omega_1)[m_1],\ldots,(A\Omega_k)[m_k]\bigr).\notag
\end{align}

The columns of $A$ are independent standard Gaussian vectors in $\R^d$. Consequently, $A\Omega_i$ is the convex hull of a standard Gaussian random walk in $\R^d$ with $n_i$ steps, and these $k$ hulls are independent, because the columns belonging to different blocks are independent. Hence Corollary~\ref{thm:discrete} applies. Since $S_{j,i}\stackrel d=\sqrt j\,G$, where $G$ is a standard Gaussian vector in $\R^d$,
\begin{equation}\label{eq:gaussian-det}
 \E\absdet{S^{(1,1)}_{j_{1,1}},\ldots,S^{(m_k,k)}_{j_{m_k,k}}}
 =\sqrt{\prod_{i=1}^k\prod_{r=1}^{m_i}j_{r,i}}\;
 \E\absdet{G_1,\ldots,G_d}
\end{equation}
with independent standard Gaussian vectors $G_1,\ldots,G_d$. The classical evaluation
\begin{equation}\label{eq:gaussian-det-constant}
 \E\absdet{G_1,\ldots,G_d}=\frac{d!\,\kappa_d}{(2\pi)^{d/2}},
\end{equation}
see, for example, \cite{KabluchkoZaporozhetsDet}, together with \eqref{eq:gaussian-det}, Corollary~\ref{thm:discrete} and \eqref{eq:spectrum-projection}, gives \eqref{eq:orthoscheme-main}.
\end{proof}

\begin{remark}\label{rem:d-equals-n}
For $d=n$ the only non-zero mixed volume in Theorem~\ref{thm:orthoscheme} has $m_i=n_i$ for every $i$. The sum then contains only the choice $j_{r,i}=1$, and \eqref{eq:orthoscheme-main} reduces to
\[
 V_n\bigl(\Omega_1[n_1],\ldots,\Omega_k[n_k]\bigr)=\frac1{n!}.
\]
This is also immediate from the geometry of the mutually orthogonal blocks, since $\Vol_n(\lambda_1\Omega_1+\cdots+\lambda_k\Omega_k)=\prod_i\lambda_i^{n_i}/n_i!$.
\end{remark}

\begin{example}\label{ex:d-equals-n-1}
Let $d=n-1$. The non-zero mixed volumes in Theorem~\ref{thm:orthoscheme} are those with $m_{i_0}=n_{i_0}-1$ for one index $i_0$ and $m_i=n_i$ for $i\neq i_0$; take, for definiteness, $i_0=1$. In the sum in \eqref{eq:orthoscheme-main}, all $j_{r,i}$ with $i\ge2$ are equal to one. The $n_1-1$ indices $j_{r,1}$ are either all equal to one, or one of them equals two and the others are equal to one. Since $\kappa_1=2$, we obtain
\[
 \E_L V_{n-1}\bigl((\Omega_1|L)[n_1-1],(\Omega_2|L)[n_2],\ldots,(\Omega_k|L)[n_k]\bigr)
 =\frac{2\kappa_{n-1}}{n!\,\kappa_n}\left(1+\frac{n_1-1}{\sqrt2}\right).
\]
Combining these mixed volumes by means of \eqref{eq:minkowski-polynomial}, we get
\[
 \E_L\Vol_{n-1}\bigl((\Omega_1+\cdots+\Omega_k)|L\bigr)
 =\frac{2\kappa_{n-1}}{n\,\kappa_n\,n_1!\cdots n_k!}
 \left(\frac{\sqrt2-1}{\sqrt2}\,n+\frac1{\sqrt2}\sum_{i=1}^kn_i^2\right).
\]
\end{example}

%======================================================================
\section*{Use of artificial intelligence tools}

In accordance with the policy of the American Mathematical Society on the use of artificial intelligence, we disclose the following. During the preparation of this manuscript the authors used Claude (Anthropic), a large language model, solely as an editing tool: to improve the language and the presentation, to check the consistency of the notation, the cross-references and the computations, and to complete the bibliography. The tool was not used to obtain any of the results. All results, ideas and proofs in this paper are the authors' own.

In particular, the discrete part of the paper was developed in its entirety during the preparation of the first author's bachelor's thesis \cite{BolotinThesis}. This part consists of the results on convex hulls of random walks in Sections~\ref{ssec:discrete-results} and~\ref{ssec:orthoscheme-results} and in Sections~\ref{sec:braid}, \ref{sec:discrete-proof} and~\ref{sec:orthoschemes}, together with their proofs. The thesis was written under the supervision of the second author at the Moscow Institute of Physics and Technology and defended in June 2025; no artificial intelligence tools were used in that work.

The authors have reviewed and edited all text prepared with the help of the tool and take full responsibility for the content of this paper.

%======================================================================


\begin{thebibliography}{99}

\bibitem{Aigner}
M.~Aigner,
\emph{Combinatorial theory},
Grundlehren der Mathematischen Wissenschaften, vol.~234, Springer, Berlin, 1979.

\bibitem{BarndorffBaxter}
O.~E. Barndorff-Nielsen and G.~Baxter,
\emph{Combinatorial lemmas in higher dimensions},
Trans. Amer. Math. Soc. \textbf{108} (1963), 313-325.

\bibitem{Baxter}
G.~Baxter,
\emph{A combinatorial lemma for complex numbers},
Ann. Math. Statist. \textbf{32} (1961), 901-904.

\bibitem{BolotinThesis}
A.~S. Bolotin,
\emph{Geometric characteristics of convex hulls of random processes},
Bachelor's thesis, Moscow Institute of Physics and Technology, Dolgoprudny, 2025 (in Russian).

\bibitem{Dospolova}
M.~K. Dospolova,
\emph{Mixed volume of infinite-dimensional convex compact sets},
Zap. Nauchn. Sem. POMI \textbf{510} (2022), 98-123 (in Russian).

\bibitem{ElBachir}
M.~El~Bachir,
\emph{L'enveloppe convexe du mouvement brownien},
Th\`ese de troisi\`eme cycle, Universit\'e Paul Sabatier, Toulouse, 1983.

\bibitem{Eldan}
R.~Eldan,
\emph{Volumetric properties of the convex hull of an $n$-dimensional Brownian motion},
Electron. J. Probab. \textbf{19} (2014), no.~45, 34~pp.

\bibitem{Freudenthal}
H.~Freudenthal,
\emph{Simplizialzerlegungen von beschr\"ankter Flachheit},
Ann. of Math. (2) \textbf{43} (1942), 580-582.

\bibitem{GaoVitale}
F.~Gao and R.~A. Vitale,
\emph{Intrinsic volumes of the Brownian motion body},
Discrete Comput. Geom. \textbf{26} (2001), 41-50.

\bibitem{GrahamKnuthPatashnik}
R.~L. Graham, D.~E. Knuth, and O.~Patashnik,
\emph{Concrete mathematics},
2nd ed., Addison-Wesley, Reading, MA, 1994.

\bibitem{Groemer}
H.~Groemer,
\emph{On the extension of additive functionals on classes of convex sets},
Pacific J. Math. \textbf{75} (1978), 397-410.

\bibitem{KVZ-GAFA}
Z.~Kabluchko, V.~Vysotsky, and D.~Zaporozhets,
\emph{Convex hulls of random walks, hyperplane arrangements, and Weyl chambers},
Geom. Funct. Anal. \textbf{27} (2017), 880-918.

\bibitem{KVZ-AdvMath}
Z.~Kabluchko, V.~Vysotsky, and D.~Zaporozhets,
\emph{Convex hulls of random walks: expected number of faces and face probabilities},
Adv. Math. \textbf{320} (2017), 595-629.

\bibitem{KabluchkoZaporozhetsDet}
Z.~Kabluchko and D.~Zaporozhets,
\emph{Random determinants, mixed volumes of ellipsoids, and zeros of Gaussian random fields},
Zap. Nauchn. Sem. POMI \textbf{408} (2012), 187-196; English transl., J. Math. Sci. (N.Y.) \textbf{199} (2014), 168-173.

\bibitem{KabluchkoZaporozhets}
Z.~Kabluchko and D.~Zaporozhets,
\emph{Intrinsic volumes of Sobolev balls with applications to Brownian convex hulls},
Trans. Amer. Math. Soc. \textbf{368} (2016), 8873-8899.

\bibitem{KampfLastMolchanov}
J.~Kampf, G.~Last, and I.~Molchanov,
\emph{On the convex hull of symmetric stable processes},
Proc. Amer. Math. Soc. \textbf{140} (2012), 2527-2535.

\bibitem{KlainRota}
D.~A. Klain and G.-C. Rota,
\emph{Introduction to geometric probability},
Cambridge University Press, Cambridge, 1997.

\bibitem{Kuhn}
H.~W. Kuhn,
\emph{Some combinatorial lemmas in topology},
IBM J. Res. Develop. \textbf{4} (1960), 518-524.

\bibitem{MajumdarComtetRandonFurling}
S.~N. Majumdar, A.~Comtet, and J.~Randon-Furling,
\emph{Random convex hulls and extreme value statistics},
J. Stat. Phys. \textbf{138} (2010), 955-1009.

\bibitem{MolchanovRandomSets}
I.~Molchanov,
\emph{Theory of random sets},
Springer, London, 2005.

\bibitem{MolchanovStable}
I.~Molchanov,
\emph{Convex and star-shaped sets associated with multivariate stable distributions. I. Moments and densities},
J. Multivariate Anal. \textbf{100} (2009), 2195-2213.

\bibitem{MolchanovWespi}
I.~Molchanov and F.~Wespi,
\emph{Convex hulls of L\'evy processes},
Electron. Commun. Probab. \textbf{21} (2016), paper no.~69, 11~pp.

\bibitem{PitmanUribeBravo}
J.~Pitman and G.~Uribe Bravo,
\emph{The convex minorant of a L\'evy process},
Ann. Probab. \textbf{40} (2012), 1636-1674.

\bibitem{RandonFurlingZaporozhets}
J.~Randon-Furling and D.~Zaporozhets,
\emph{Convex hulls of several multidimensional Gaussian random walks},
Zap. Nauchn. Sem. POMI \textbf{505} (2021), 244-281.

\bibitem{Sato}
K.~Sato,
\emph{L\'evy processes and infinitely divisible distributions},
Cambridge University Press, Cambridge, 1999.

\bibitem{Schneider}
R.~Schneider,
\emph{Convex bodies: the Brunn-Minkowski theory},
2nd expanded ed., Cambridge University Press, Cambridge, 2014.

\bibitem{SchneiderWeil}
R.~Schneider and W.~Weil,
\emph{Stochastic and integral geometry},
Springer, Berlin, 2008.

\bibitem{Spitzer}
F.~Spitzer,
\emph{A combinatorial lemma and its application to probability theory},
Trans. Amer. Math. Soc. \textbf{82} (1956), 323-339.

\bibitem{SpitzerWidom}
F.~Spitzer and H.~Widom,
\emph{The circumference of a convex polygon},
Proc. Amer. Math. Soc. \textbf{12} (1961), 506-509.

\bibitem{Stanley}
R.~P. Stanley,
\emph{Enumerative combinatorics. Vol. 1},
2nd ed., Cambridge University Press, Cambridge, 2012.

\bibitem{Steele}
J.~M. Steele,
\emph{The Bohnenblust-Spitzer algorithm and its applications},
J. Comput. Appl. Math. \textbf{142} (2002), 235-249.

\bibitem{Sudakov}
V.~N. Sudakov,
\emph{Geometric problems in the theory of infinite-dimensional probability distributions},
Trudy Mat. Inst. Steklov. \textbf{141} (1976), 3-191 (in Russian).

\bibitem{Tsirelson}
B.~S. Tsirelson,
\emph{A geometric approach to maximum likelihood estimation for an infinite-dimensional Gaussian location. II},
Theory Probab. Appl. \textbf{30} (1985), 820-828.

\bibitem{Vitale}
R.~A. Vitale,
\emph{Expected absolute random determinants and zonoids},
Ann. Appl. Probab. \textbf{1} (1991), 293-300.

\bibitem{Volland}
W.~Volland,
\emph{Ein Fortsetzungssatz f\"ur additive Eipolyederfunktionale im euklidischen Raum},
Arch. Math. \textbf{8} (1957), 144-149.

\bibitem{VysotskyZaporozhets}
V.~Vysotsky and D.~Zaporozhets,
\emph{Convex hulls of multidimensional random walks},
Trans. Amer. Math. Soc. \textbf{370} (2018), 7985-8012.

\end{thebibliography}
\end{document}